\documentclass[11pt]{article}
\usepackage{amssymb}
\usepackage{amsmath}
\usepackage{graphicx}
\usepackage{amsmath}
\usepackage{epsfig}
\usepackage{float}
\usepackage{wrapfig}
\usepackage{hyperref}
\usepackage{tikz}
\usepackage{amscd}
\usetikzlibrary{matrix}
\usepackage{ushort}
\usetikzlibrary{decorations.markings}
\pgfdeclarelayer{edgelayer}
\pgfdeclarelayer{nodelayer}
\pgfsetlayers{edgelayer,nodelayer,main}
\usepackage[scanall]{psfrag}
\newtheorem{theorem}{Theorem}

\newtheorem{convention}[theorem]{Convention}

\newtheorem{corollary}[theorem]{Corollary}

\newtheorem{definition}[theorem]{Definition}
\newtheorem{example}[theorem]{Example}

\newtheorem{proposition}[theorem]{Proposition}
\newtheorem{remark}[theorem]{Remark}

\newenvironment{proof}[1][Proof]{\textbf{#1.} }{\ \rule{0.5em}{0.5em}}

\newcommand {\id}{\mathrm{id}}

\renewcommand {\O}{{\cal O}}

\newcommand {\Hom}{\mathrm{Hom}}
\newcommand{\GL}{\mathrm{GL}}

\newcommand{\tr}{\mathrm{tr}}

\newcommand{\C}{\mathbb{C}}
\newcommand{\R}{\mathbb{R}}
\newcommand{\Z}{\mathbb{Z}}
\newcommand{\Q}{\mathbb{Q}}

\newcommand{\gX}{\mathrm{X}}
\newcommand{\Maps}{\mathrm{Maps}}

\newcommand{\cG}{\check{G}}
\newcommand{\cH}{\check{H}}
\newcommand{\cM}{\check{M}}
\newcommand{\cN}{\check{N}}
\newcommand{\calA}{\mathcal{A}}
\newcommand{\ccalA}{\check{\mathcal{A}}}
\newcommand{\calB}{\mathcal{B}}
\newcommand{\ccalB}{\check{\mathcal{B}}}
\newcommand{\rG}{\mathrm{G}}
\newcommand{\setA}{\mathrm{A}}
\newcommand{\setB}{\mathrm{B}}

\newcommand{\Tr}{\mathrm{Tr}}

\begin{document}
\title{Duality for finite Gelfand pairs}
\author{M.V. Movshev\footnote{The author would like to thank ICTP-SAIFR and FAPESP grants 2016/07944-2 and 2016/01343-7 for partial financial support.
}\\
Mathematics Department, Stony Brook University, USA and\\
ICTP South American Institute for Fundamental Research,\\
IFT-UNESP, S\~{a}o Paulo State University, Brazil \\
\texttt{mmovshev@math.sunysb.edu}}
\date{\today}
\maketitle

\begin{abstract}
Let $\rG$ be a split reductive group, $K$ be a non-Archimedean local field, and $O$ be its ring of integers. Satake isomorphism identifies the algebra of compactly supported invariants 
$\C_c[\rG(K)/\rG(O))]^{\rG(O)}$ with a complexification of the algebra of characters of finite-dimensional representations $\O(\rG^L(\C))^{\rG^L(\C)}$ of the Langlands dual group.
In this note we report on the results of the study of analogues of such an isomorphism for finite groups. In our setup we replaced Gelfand pair $\rG(O)\subset \rG(K)$ by a finite pair $H\subset G$. It is convenient to rewrite the character side of the isomorphism as $\O(\rG^L(\C))^{\rG^L(\C)}=\O((\rG^L(\C)\times \rG^L(\C))/\rG^L(\C))^{\rG^L(\C)}$. We replace diagonal  Gelfand pair $\rG^L(\C)\subset \rG^L(\C)\times \rG^L(\C)$ by a dual finite pair $\check{H}\subset \check{G}$ and use Satake isomorphism as a defining property of the duality. In this text we make a preliminary study of such duality and compute a number of nontrivial examples of dual pairs $(H,G)$ and $(\check{H}, \check{G})$. We discuss a possible relation of our constructions to String Topology.
 
\end{abstract}
\tableofcontents

\section{Introduction}
Langlands duality for reductive groups and Satake isomorphism are important concepts that are used in the formulation of Langlands conjectures. The author's motivation for writing this paper is placing these concepts into a more general context with the hope that in the future it might reveal new perspectives on the classical Langlands correspondence.

Let us recall (see \cite{FrenkelLanglands} for an accessible  introduction to Langlands theory) some basic definitions needed to state Satake isomorphism. We write $\rG$ for a split reductive group, $K$ for a non-Archimedean local field, and $O$ for its ring of integers. By following the idea mentioned in the abstract we are going to replace the linear space $\C_c[ \rG(K)/\rG(O))]^{\rG(O)}=\C_c[\rG(O)\backslash \rG(K)/\rG(O))]$ of Hecke operators by $\C[H\backslash G/H]$ with finite $G$ and $H$. As usual, a group in superscript $V^G$ signals for taking $G$-invariants. Similarly the space of complexified algebraic characters $\O[\rG^L(\C)]^{\rG^L(\C)}=\O[\rG^L(\C) \backslash \rG^L(\C)\times \rG^L(\C) /\rG^L(\C)]$\footnote{$\rG^L(\C)$ is embedded diagonally into $\rG^L(\C))\times \rG^L(\C)$} of the Langlands dual group $\rG^L(\C)$ over $\C$ in our setup has a finite analogue $\C[\cH\backslash \cG/\cH]$. In general, we don't assume that $\cG$ splits into a product just like $\rG^L(\C)\times \rG^L(\C)$.

We would like to think about elements of $\C[H \backslash G/H]\subset \C[G]$ as $H$-bi invariant functions on $G$: $f(hgh')=f(g),g\in G, h,h'\in H$. The linear space $\C[H \backslash G/H]$ has two unital algebra structures. The first product $f\cdot g$ is a point-wise multiplication of functions. 
The second is the convolution $f\times g$, which in general is noncommutative. Recall that a pair of finite groups $H\subset G$ is a {\it Gelfand pair} iff $f\times g$ is commutative.
Roughly speaking, we want to study duality on Gelfand pairs 
\begin{equation}\label{E:correspondence}
H\subset G\Leftrightarrow \check H\subset \check G
\end{equation} which interchanges the products 
\begin{equation}\label{E:swap1}
\cdot \Leftrightarrow \times.
\end{equation} We don't expect $\check H\subset \check G$ to exist for an arbitrary Gelfand pair $H\subset G$. Neither do we hope that $\check H\subset \check G$ will be given by a functorial construction. A notable exception is $H=\{1\}$ and $G$ is abelian. Then $\check H=\{1\}$ and $ \check G=\Hom_{\Z}(G,\C^{\times})$.

One of the reasons for studying such an amorphous structure (besides formal resemblance with Satake isomorphism) is a link it provides between combinatorics and algebra. We postpone, for now, the formal definition of duality. Instead we formulate its corollary, which will be proven in due time (see (\ref{E:orbitrep})). To state it we notice that the space $X=G/H$ splits into a disjoint union $X=\bigcup o_i$ of $H$-orbits (in this paper we assume that the map $G\to Aut(X)$ has no kernel). The space of functions $\C[X]$ as a regular $G$-representation decomposes into a direct sum of irreducible multiplicity-one representations $\C[X]=\bigoplus_i T_i$. By $a(H,G),b(H,G)$ we denote the arrays $\{\#o_i\}, \{\dim T_i\}$, respectively. One of the nice features of (\ref{E:correspondence}) is that it swaps the sizes of orbits, which live on combinatorial side of the correspondence, and the dimensions of representations, which belong to algebraic side: 
\begin{equation}\label{E:swaporbrep}
a(H,G)=b(\check H,\check G),b(H,G)=a(\check H, \check G)\text{ and } |X|=|\check X|.
\end{equation}
\paragraph{An illustration}\label{P:illustration}
It is amusing to see that $(H,G)=(\Z_4,S_4)$ and $(\check H,\check G)=(S_3, S_3\times S_3)$ is an example of pairs $(H,G)\not\cong (\check H,\check G)$ which satisfy (\ref{E:swaporbrep}). In fact, it is the simplest example of a non self-dual, noncommutative pairs in duality (\ref{E:correspondence}).

In the second pair, $S_3$ is embedded diagonally into $S_3\times S_3$. The orbits of $S_3$ in $S_3\times S_3/S_3$ coincide with the conjugasy classes in $S_3$. These are \[\{1\}, \{(1,2),(1,3),(2,3)\}, \{(1,2,3),(1,3,2)\}.\] So $a(S_3,S_3\times S_3)=\{1,2,3\}$. The trivial, sign, and two-dimensional representations exhaust the set $Irrep(S_3)$ of irreducible representations of $S_3$. By Peter-Weyl isomorphism 
\begin{equation}\label{E:PeterWeyl}
\C[G]\cong \bigoplus_{T\in Irrep(G)} T\boxtimes \overline{T}.
\end{equation}
$b(S_3,S_3\times S_3)=\{1,1,4\}$.

The group of rotations of a cube in $\R^3$ is isomorphic to $S_4$ (it coincides with the group of permutations of diagonals). 
Let $X$ be the set of faces of the cube. The stabilizer of a face is isomorphic to $\Z_4$. Thus $X\cong S_4/\Z_4$.
Under $\Z_4$ $X$ breaks into a union of two one-point orbits and one four-point orbit. We infer that $a(\Z_4,S_4)=\{1,1,4\}$. Frobenius formula
\[\Hom_H(Res^{G}_H T_{i},\C)\cong \Hom_G(T_{i},Ind_H^{G} \C)=\Hom_G(T_{i},\C[X])=V_i
\]
relates $\Z_4$ invariants in $T_i$ with the multiplicities $V_i$ of $T_i$ in $\C[X]$. As usual, $Ind_H^{G} \C$ stands for the $G$-representation induced from the trivial representation of subgroup $H$, and $Res^{G}_H T$ stands for the restriction of $T$ from $G$ to $H$. Thus all representations $T_i$ that contain the $\Z_4$-invariant vectors are subrepresentations of $\C[X]$.
 Among such is the tautological 3-dimensional representation twisted by $sign(\sigma)$.
In order to construct another subrepresentation of $\C[X]$ we notice that $S_4$ can be mapped onto $S_3$. Geometrically this homomorphism comes from the $S_4$ action on pairs of opposite edges of the tetrahedron. Under this map a generator $e$ of $\Z_4$ maps to a reflection in $S_3$. Because of this, the two-dimensional representation of $S_3$ 
 pulled back to $S_4$ contains a $\Z_4$-invariant vector. 
Constants define a trivial subrepresentation in $\C[X]$. Note that $\dim \C+\dim \C^2+\dim \C^3=\dim \C[X]=6$. We conclude that $b(\Z_4,S_4)=\{1,2,3\}$.
We see that 
\[a(S_3,S_3\times S_3)=b(\Z_4,S_4),\quad b(S_3,S_3\times S_3)=a(\Z_4,S_4).\]
For more examples the reader can consult Section \ref{S:Examples}.

\paragraph{Algebras with two multiplicative structures}
Linear spaces $\C[H\backslash G/H]$ have more structures then multiplications $\cdot$ and $\times$ mentioned above. We are going to enhance $\C[H\backslash G/H]$ by these structures and put the resulting object inside of a certain category. This will let us state in more precise terms what it means to swap the products in formula (\ref{E:swap1}).

We will be writing $X$ for $G/H$. We start with an observation that 
\begin{equation}\label{E:simpleiso}
\text{{\it the space $\C[X\times X]^G$ is isomorphic to $\C[X]^H\cong \C[H\backslash G/H]$ }:}
\end{equation} $r:f(x,x')\to f(x,x_0'),$ with $St(x_0)$ equal to $H$. The two products $\times, \cdot$ on $\C[H\backslash G/H]$ and other structures 
are easier to describe in terms of $\C[X\times X]^G$. This what we are going to do now.

To this end, we start in a greater generality by fixing finite sets $X,Y$, which at the moment have no relation to the pair $(H,G)$. 
$\C[X\times Y]$ is an algebra with respect to the point-wise multiplication of functions
\begin{equation}\label{Def:1}
(f\cdot g)(x,y):= f(x,y) g(x,y)\text{ with the unit } 1_{\cdot}(x,y)=1.
\end{equation}

The convolution product $\C[X\times Y]\times \C[Y\times Z]\to \C[X\times Z]$ 
is given by
\begin{equation}\label{Def:2} 
(f\times g)(x,y):= \sum_{y\in X}f(x,y) g(y,z),\quad f\in \C[X\times Y], g\in \C[Y\times Z].\end{equation}
 $\C[X\times Y]$ is isomorphic to the space of linear maps $\Hom_{\C}(\C[Y],\C[X])$. The isomorphism is defined by the formula
\begin{equation}\label{Def:25}
(f\times h)(x):= \sum_{y\in Y}f(x,y) h(y), \quad h\in \C[Y].
\end{equation}
Under this identification, the convolution becomes a composition of maps. In particular, $\C[X\times X]$ is an algebra $\text{ with unit } 1_{\times}(x,y)=\delta_{xy} x,y\in X$
and $\C[X\times Y]$ is a $\C[X\times X]$, $\C[Y\times Y]$ bimodule. 
$\C[X\times X]$ is equipped with two functionals:
\begin{equation}\label{Def:2cdottimes} 
\tr_{\cdot}(f):= \sum_{(x,y)\in X\times X}f(x,y)\quad \tr_{\times}(f):= \sum_{x\in X}f(x,x)
\end{equation}

 \begin{equation}\label{E:dotcomp}
 (a,b):=\tr_{\cdot}(a\cdot \pi b)= \tr_{\times}(a\times \mu b)
 \end{equation} 
 \begin{equation}\label{E:positive}
 \text{ and is positive-definite: }(a,a)>0, a\neq0.
 \end{equation} 

The units satisfy
\begin{equation}\label{Def:26} 
1_{\cdot} \times1_{\cdot}=|X|1_{\cdot}\quad 1_{\times} \cdot1_{\times}=1_{\times}
\end{equation}
We are going to use to complex anti-linear involutions on $\C[X\times X]$. The first one is the complex conjugation map 
\begin{equation}\label{E:congc}
\pi(f)(x,y)= \bar{f}(x,y).
\end{equation} 
The second is a Hermitian conjugation 
\begin{equation}\label{E:transp}
(\mu f)(y,x):=\bar{f}(x,y).
\end{equation} 
Note that 
\begin{equation}\label{E:commutativity}
\pi\mu=\mu\pi
\end{equation}
and 
 \begin{equation}\label{E:realtrace}
 \begin{split}
& \tr_{\cdot}(\pi a)=\overline{\tr_{\cdot}(a)}, \tr_{\cdot}(\mu a)=\overline{\tr_{\cdot}(a)},\\
& \tr_{\times}(\pi a)=\overline{\tr_{\times}(a)}, \tr_{\times}(\mu a)=\overline{\tr_{\times}(a)}.
 \end{split}
 \end{equation}
Traces are related by the identities
\begin{equation}\label{Def:256} 
\tr_{\cdot}(f)=\tr_{\times}(f\times 1_{\cdot}),\quad \tr_{\times}(f)=\tr_{\cdot}(f\cdot 1_{\times}).
\end{equation}

Suppose now that $X=Y=Z=G/H$ for the finite $H\subset G$. All of the above structures are compatible with the regular left $G$ action on $\C[X\times X]$. This way $\C[X\times X]^G$ inherits the products $\cdot$ and $\times$, the traces $\tr_{\cdot}$ and $\tr_{\times}$, and the involutions $\pi$ and $\mu$.

Our basic objects will be Frobenius algebras with some additional data:
\[\begin{split}
&U:=(\C[X\times X]^G,\cdot,\tr_{\cdot}(1_{\cdot})^{-1}\tr_{\cdot}, \pi_{\cdot}, \mu_{\cdot})\\
& \pi_{\cdot}:=\pi, \mu_{\cdot}:=\mu
\end{split}\] 
 and 
 \[\begin{split}
&W:=(\C[X\times X]^G,\times,\tr_{\times}(1_{\times})^{-1}\tr_{\times}, \pi_{\times}, \mu_{\times})\\
& \pi_{\times}=\mu, \mu_{\times}=\pi
\end{split}
\] 
There is a tautological pairing between linear spaces $U$ and $W$ 
\begin{equation}
\begin{split}
&\langle a, b \rangle:=\tr_{\cdot}(a\cdot \pi(b))
, a\in U, b\in W\\
&\text{which satisfies } \langle a, \mu_{\times}b \rangle= \overline{\langle \pi_{\cdot}a, b \rangle}, \langle a, \pi_{\times}b \rangle= \overline{\langle \mu_{\cdot}a, b \rangle}
\end{split}
\end{equation}

To summarize, we have a construction
\[H,G\Rightarrow \calA(H,G) \]
 where 
 \begin{equation}\label{E:quadruple}
 \calA\text{ is the triple }(U,W,\langle \cdot,\cdot \rangle).
 \end{equation}
 Define a new triple $\ccalA$ as $(W,U,\overline{\langle\cdot,\cdot \rangle})$ by interchanging $U$ and $W$ (c.f. this with the definition from Section 3 in   \cite{Egge}).
 \begin{remark}\label{R:catTdef}
 {\rm
 We want to think about triples $(W,U, \langle\cdot,\cdot \rangle)$ as objects of a category $\mathcal{T}$. In $\mathcal{T}$ we don't insist on 
 $\tr_U$ and $\tr_W$
 being nonzero, nor on $\dim U$ and $\dim W$ being finite. Still we require $ \langle\cdot,\cdot \rangle$ to be a nondegenerate pairing.
The morphisms in $\mathcal{T}$ are homomorphisms (inclusions) of underlying algebras compatible with inner products involutions and traces. $\mathcal{T}$ is a monoidal category with respect to the tensor product of triples.
}\end{remark}
We would like to explore the following problem:
Under what conditions on the pair $(H,G)$ is there a dual pair $(\cH,\cG)$, such that there is an isomorphism in $\mathcal{T}$
\begin{equation}\label{E:Satakegen}
D:\calA(H,G)\cong\ccalA(\cH,\cG),
\end{equation}
and what is the freedom in the choice of $(\cH,\cG)$? 

Notice that if we formally replace $H,G$ in (\ref{E:Satakegen}) by $\rG(O),\rG(K)$ and $\cH,\cG$ by $ \rG^L(\C),\rG^L(\C)\times \rG^L(\C)$ the classical Satake isomorphism identifies $W(H,G)$ and $U(\check{H},\check{G})$. It is part of the structures needed for 
 (\ref{E:Satakegen}) to hold. Of course, in the case of infinite groups one has to be careful about analytic aspects, which will be ignored in our finite setting.
\begin{remark}\label{R:com}{\rm
The algebra $\C[H\backslash G/H]$ acts on $\C[ G/H]$ from the right by endomorphisms of $G$-regular representation in $\C[ G/H]$. This observation is used in the proof of the criteria Lemma 3.9. p.42 \cite{Kirillov}:
$(H,G)$ is a Gelfand pair $\Leftrightarrow$ 
\begin{equation}\label{E:induceddec}
\C[ G/H]=\bigoplus_i T_i\otimes V_i(G/H), T_i\neq T_j, \dim V_i\leq 1.
\end{equation} In this formula $T_i$ stands for irreducible $G$-representations, and $\dim V_i(G/H)$ for their multiplicities.
}
\end{remark}
The most well-known example of a Gelfand pair is $G$ embedded diagonally into $ G\times G$. This follows, for example, from Remark \ref{R:com} and Peter-Weyl theorem \cite{Pontryagin}.

\paragraph{Relation to topology} It turns out that structures similar to $\calA(H,G)$ appear in topology. To see this, we fix a connected (for simplicity) oriented submanifold $N$ in a finite-dimensional compact connected manifold $ M$. Let us define the space of paths
\[L(N,M):=\left\{\gamma:[0,1]\to M|\gamma(0),\gamma(1)\in N\right\}.\]
The first indication that $L(N,M)$ has some relation to previous constructions is that the set of connected components $\pi_0(L(N,M))$ is isomorphic to \\ $i(\pi_1(N))\backslash \pi_1(M)/i(\pi_1(N))$ (see e.g. \cite{RonaldBrown} p.397 Corollary 10.7.6). The map $i$ stands for embedding $N\to M$. In order to simplify the notations we will often omit $i$ and the base point in the formulas. 

The space $\C[\pi_0(L(N,M))]=\C[\pi_1(N)\backslash \pi_1(M)/\pi_1(N)]$ is isomorphic to zero cohomology $H^{0}(L(N,M),\C)$ and zero homology $H_{0}(L(N,M),\C)$. 

These linear spaces are parts of a richer structure. The direct sum $\mathbb{H}^{*}(N,M)=\bigoplus_{i\geq 0} H^{i}(L(N,M),\C)$ is the graded commutative algebra with respect to the $\cup$-product. It has two involutions $\pi(a)=\bar a$. Involution $\mu$ is a composition of complex conjugation $\pi$ and geometric operation $or^*$ that changes orientation of a path $or(\gamma)(t)=\gamma(1-t)$.

Homology groups $\mathbb{H}_{*}(N,M):=\bigoplus_{i\geq 0} H_{i}(L(N,M))$ has an algebra structure \[\mathbb{H}_{i}\otimes \mathbb{H}_{j}\to \mathbb{H}_{i+j-\dim N}\] defined by concatenation of paths (see \cite{Sullivan2},\cite{HingstonandOancea}). 
This product is denoted by $a\bullet b$.

By definition a developable orbifold $M$ is quotient of a manifold $\widetilde{M}$ by the discrete group $\Gamma$ acting properly. In this case $\pi_1(M)$ by definition is equal to $\Gamma$. In our application we choose $\widetilde{N}$ to be a $\Gamma$-invariant submanifold. 

\begin{example}\label{Ex:orb}To make connection with the algebraic discussion we take $\widetilde{M}$ to be a compact simply connected manifold equipped with a finite group $G$ action. For $\widetilde{N}$ we take a $G$-orbit of a point $y\in \widetilde{M}$. Denote by $p:\widetilde{M}\to {M}$ the canonical projection. The submanifold $N\subset M$ is an orbifold point $\{x\}=p(\widetilde{N})$. In this setup $\pi_1(M):=G$ and $\pi_1(N):=St(y)=H$. Now $\widetilde{N}$ is disconnected but the appropriate modifications of the previous constructions go through.

The product on $\mathbb{H}_{*}(N,M)$ extends without troubles to developable orbifolds (see \cite{LupercioUribeaXicotencatl} for $M\subset M\times M$ case). 
The linear space $\mathbb{H}_{0}(\{x\},M)$ with the product $\times$ is the Hecke algebra of the pair of groups $(St(y), G)$. The space $\mathbb{H}^{0}(\{x\},M)$ with $\cup$-product coincides with $(\C[St(y)\backslash G/St(y)],\cdot)$.
\end{example}

Now we would like to assume that $N$ is an even-dimensional manifold. We reduce the grading in algebras $\mathbb{H}_{*}(N,M), \mathbb{H}^{*}(N,M)$ modulo two. This way we get a triple 
\begin{equation}\label{E:toptriple}
\begin{split}
&\mathcal{B}(N,M)=(U, W,\langle\cdot,\cdot\rangle),\\
&U=(\mathbb{H}^{*}(N,M),\cup,\tr=0,\pi_{\cup},\mu_{\cup}), \pi_{\cup}:=\pi,\mu_{\cup}=\mu,\\
&W=(\mathbb{H}_{*}(N,M),\bullet,\tr=0,\pi_{\bullet},\mu_{\bullet}),\pi_{\bullet}=\mu,\mu_{\bullet}=\pi.\\
\end{split}
\end{equation}In the above formulas $\langle\cdot,\cdot\rangle$ stands for the pairing between homology and cohomology. Had $L(N,M)$ been a finite-dimensional manifold, we could have identified homology with cohomology by the Poincare duality pairing. That would have given us a triple $\calA$ similar to the one appearing in the group case. In reality, $L(N,M)$ is a infinite-dimensional space. Thus the triple $\calB$ is a substitute for $\calA$ in infinite dimensions. There is a duality operation acting on a triple $\calB$ which interchange $\mathbb{H}_{*}(N,M)$ and $\mathbb{H}^{*}(N,M)$. Denote the result of such modification by $\ccalB$.

We say that a pair of manifolds $\cN\subset \cM$ is dual to $N\subset M$ if
\[\calB(N, M)\cong \ccalB(\cN, \cM).\]
As in the case with groups, in order for this isomorphism to hold, $\mathbb{H}_{*}(N,M)$ must be a commutative algebra. In this case we will refer to the pair $(N, M)$ as a {\it topological Gelfand pair}.

As in the group case, the diagonal embedding $M\to M\times M$ of orientable $M$ gives an example a topological Gelfand pair. Indeed the space $L(M,M\times M)$ is homeomorphic to the free loop space $L(M)=\Maps(S^1,M)$. (A pair of restrictions of $\gamma\in L(M)$ on semicircles $S^1_+\cup S^1_-=S^1$ gives an element in $L(M,M\times M)$.) The $\times$ product on $\mathbb{H}_{*}(M,M\times M)$ coincides with the $\bullet$ product $H_{*}(L(M))$, which is commutative (see \cite{ChasSullivan} for details).

It is not obvious that the proposed duality is nontrivial. In the following sections we will see that there are plenty of examples of dual group pairs. 

Here is an outline of the paper. The text is divided roughly in two parts: theoretical material (Section \ref{S:Dualityforgrouppairs}) and a collection of examples (Section \ref{S:Examples}). Here is a more detailed breakdown.

In Section \ref{S:Dualityforgrouppairs} we study at some length the general aspects of the duality. Thus in Section \ref{S:Dualityinabstractterms} we isolate features of duality that can be formulated without mentioning Gelfand pairs. In Section \ref{S:matrixorbits} we discuss properties of the duality related to the combinatorics of $G$-action on $G/H$. Section \ref{S:zonal} contains, besides the well-known material about spherical functions, the formulas for idempotents for $\times$-product. The general formula for matrix $C$ is derived in Section \ref{S:matrixC}. Some algebraic and arithmetic properties of the coefficients of matrix $C$ are determined in Section \ref{S:Arithmetic}. Based on this we formulated a necessary condition for existence of the dual pair $\check{H}, \check{G}$ in Section \ref{S:Dualityisomorphism}. Conjectural relation of the Galois group of the splitting field and duality is discussed in Section \ref{S:Galoisgroupandduality}. In Section \ref{S:MapsbetweenGelfandpair} we set up a language for studying (non)uniqueness of the dual pair. Section \ref{S:Examples} is devoted to examples of dual pairs described with various levels of precision. In Sections \ref{S:SN}, \ref{S:wr}, and \ref{S:semidirect} we manually compute the triples $(A,B,C)$. In Section \ref{S:Nonexample} we illustrate how a necessary condition established in previous sections work in the simplest example. Rudiments of computed-aided classification of dual pairs are presented in Section \ref{S:Onclassification}. In Section \ref{S:concluding} we listed some of the open problems in the outlined theory.
 
\paragraph{Acknowledgment}
The author benefited from conversations on the subject of this paper with P. Bressler, P. Deligne, S. Donaldson, O. Gabber, D. Kaledin, M. Kontsevich, A. Kuznetsov, L. Mason, R. Matveyev, A. Mikhailov, A.S. Schwarz, D. Sullivan, L. Takhtajan,  O.Viro. 
The author owes special thanks  to  P.  Terwilliger who has pointed at \cite{Egge}. That work contains much overlap with  Section \ref{S:Dualityinabstractterms}. 
Part of the work on this project was conducted at Mathematical Institute  University of Oxford, ICTP (S\~{a}o Paulo), and IHES. The author would like to thank these institutions for excellent working conditions. 

\section{Duality for group pairs}\label{S:Dualityforgrouppairs}
We will see in this section that a triple $\calA(H,G)$ (\ref{E:quadruple}) can be effectively encoded by a square matrix $C(H,G)$. In terms of this matrix, the duality corresponds to taking the Hermitian adjoint $C(H,G)^{\dagger}$.

\subsection{Duality in abstract terms}\label{S:Dualityinabstractterms}
 
 The goal of this section is to give a more economic description for the triple (\ref{E:quadruple}). For this purpose we take a linear space $Q$ as a prototype for $\C[H\backslash G/H]$.
We assume that $Q$ is equipped with two commutative algebra structures $\cdot$ and $\times$, has two traces $\tr_{\cdot}$ and $\tr_{\times}$, has two commuting (\ref{E:commutativity}) anti-linear involution $\pi,\mu$ compatible with both multiplications. We require that traces are real (\ref{E:realtrace}).
 They also satisfy (\ref{E:dotcomp}),(\ref{E:positive}). Note that equations (\ref{Def:256}) are formal corollaries of (\ref{E:dotcomp}) and (\ref{E:commutativity}).
 
We demand that units satisfy (\ref{Def:26}).
 Denote by $\calA(Q)$ the triple $(U(Q),W(Q),\langle a, b \rangle)$ where 
 \begin{equation}\label{E:triple}
 \begin{split}
 &U(Q):=(Q,\cdot, \tr_{\cdot}(1_{\cdot})^{-1}\tr_{\cdot},\pi,\mu),\\
& W(Q):=(Q,\times, \tr_{\times}(1_{\times})^{-1}\tr_{\times},\mu,\pi),\\
& \langle a, b \rangle=(a,b),a\in U,b\in W.
 \end{split}
 \end{equation}

We fix notations for normalized traces: \[\Tr_U:=\tr_{\cdot}(1_{\cdot})^{-1}\tr_{\cdot}, \Tr_W:=\tr_{\times}(1_{\times})^{-1}\tr_{\times}.\]
\begin{definition}\cite{Egge} {\rm
 The septuple $(Q,\cdot,\times,\tr_{\cdot}, \tr_{\times}, \pi,\mu)$ is called  a character algebra (c.f. Definition 2.1 {\it loc.cit}).
}\end{definition}
\begin{proposition}\label{P:equiv}
 Under above assumption the triple $\calA(Q)$ (\ref{E:triple}) determines a character algebra $(Q,\cdot,\times,\tr_{\cdot}, \tr_{\times}, \pi,\mu)$.
 \end{proposition}
 \begin{proof}
 The only difference between $\calA(Q)$ and $(Q,\cdot,\times,\tr_{\cdot}, \tr_{\times}, \pi,\mu)$ is that of normalization of traces in $\calA(Q)$. As soon as we recover the normalization constants $\tr_{\times}(1_{\times})$, $\tr_{\cdot}(1_{\cdot})$ we recover the septuple.
 
From nondegeneracy of $\Tr_U,\Tr_W$ (\ref{E:positive}) we derive existence of elements $e_U\in U,e_W\in W$ such that
\[\langle a, 1_{\times}\rangle =\Tr_U(a\cdot e_U), a\in U,\quad \overline{\langle 1_{\cdot },b\rangle} =\Tr_W(e_W\times b), b\in W.\]
It follows from (\ref{Def:256}) that $e_U=\tr_{\cdot}(1_{\cdot})1_{\times} \in U$ and $e_W=\tr_{\times}(1_{\times})1_{\cdot} \in W$. The $\calA(Q)$-dependent constants $K,N$ are defined by the formulas
\[K:=\langle e_U, 1_{\times} \rangle=\langle 1_{\cdot}, e_W\rangle=\tr_{\cdot}(1_{\cdot})\tr_{\times}(1_{\times}),\] 
\[N:=\langle 1_{\cdot}, 1_{\times} \rangle=\tr_{\times}(1_{\cdot})=\tr_{\times}(1_{\cdot}\times \mu1_{\times})\overset{(\ref{E:dotcomp})}{=}\tr_{\cdot}(1_{\cdot}\cdot \pi1_{\times})=\tr_{\cdot}(1_{\times}).\]
Units $1_{\cdot},1_{\times}$ are obviously invariant with respect to $\pi$ and $\mu$. We have the following (in)equalities:
\[0\overset{(\ref{E:positive})}{<}\tr_{\cdot}(1_{\cdot})=\tr_{\cdot}(1_{\cdot}\cdot \pi1_{\cdot})\overset{(\ref{E:dotcomp})}{=}\tr_{\times}(1_{\cdot}\times \mu 1_{\cdot})\overset{(\ref{Def:26})}{=}|X|\tr_{\times}(1_{\cdot})=|X|N,\]
\[0\overset{(\ref{E:positive})}{<}\tr_{\times}(1_{\times})=\tr_{\times}(1_{\times}\times \mu1_{\times})\overset{(\ref{E:dotcomp})}{=}\tr_{\cdot}(1_{\times}\cdot \pi 1_{\times})\overset{(\ref{Def:26})}{=}\tr_{\cdot}(1_{\times})=N.\]
We conclude that 
\[\tr_{\cdot}(1_{\cdot})=K/N,\quad \tr_{\times}(1_{\times})=N,\quad \tr_{\cdot}(1_{\times})=\tr_{\times}(1_{\cdot})=N,\quad |X|=K/N^2.\]

\end{proof}
\begin{proposition}\label{P:bases}
\begin{enumerate}
\item In a character algebra $E=(Q,\cdot,\times,\tr_{\cdot}, \tr_{\times}, \pi,\mu)$  $(Q,\cdot)$ and $(Q,\times)$ are isomorphic to direct products $\C^{\dim Q}$ of fields. 
 \item We fix an inner product as in (\ref{E:dotcomp}) and two sets $\setA,\setB$, $|\setA|=|\setB|=\dim Q$. Up to an isomorphism $E$ is completely characterized by two real vectors $A\in \C^{\setA},B\in \C^{\setB}$ and a matrix $C$ such that 
\begin{equation}\label{E:Adef}
(\gX_i,{\gX}_j)=A_{i}\delta_{ij},A_i=\tr_{\cdot}\gX_i, A_{i}>0,
\end{equation}
\begin{equation}\label{E:Bdef}
(\Psi_i,{\Psi}_j)=B_{i}\delta_{ij},B_i=\tr_{\cdot}\Psi_i, B_{i}>0,
\end{equation}
\begin{equation}\label{E:Cdef}
(\gX_i,{\Psi}_j)=C_{ij}. 
\end{equation}
$\{\gX_i|i\in \setA\}$ is the basis of minimal idempotents defined with respect to multiplication $\cdot$. $\{\Psi_j | j\in \setB\}$ is the similar basis for $\times$-multiplication. 
There are involutions $\mu:\setA\to\setA$, $\pi:\setB\to\setB$ such that
 \begin{equation}\label{E:invact}
 A_{i}=A_{\mu(i)}, B_{i}=B_{\pi(i)}, C_{i\pi(j)}=\overline{C}_{ij}, C_{\mu(i)j}.=\overline{C}_{ij}
 \end{equation}

The isomorphism of two triples $(A,B,C), (A',B',C')$ is a pair of permutations $\sigma,\tau$ of indices $\setA,\setB$ such that 
$A_{\sigma(i)}=A'_{i}, B_{\tau(j)}=B'_{j}, C_{\sigma(i)\tau(j)}=C'_{ij}, i\in \setA, j\in \setB$. $\sigma$ intertwines $\mu$ and $\mu'$. $\tau$ intertwines $\pi$ and $\pi'$.

The data satisfy 
\begin{equation}\label{E:compatibility}
B_{k}\delta_{kl}=\sum_{i=1}^{\dim Q}A^{-1}_{i}C_{ik} \overline{C}_{il}.
\end{equation}
\item \label{E:idempones} By (\ref{Def:26}) $1_{\cdot}/|X|$ is an idempotent for $\times$-multiplication, $1_{\times}$ - for $\cdot$-multiplication. Thus 
\[1_{\cdot}/|X|=\sum_in_i\Psi_i,\quad 1_{\times}=\sum_im_i\gX_i,n_i,m_i=0,1. \]
Then
\begin{equation}\label{E:cacbrelgen}
A_i=|X|\sum_jC_{ij}n_j,\quad B_j=\sum_jC_{ij}m_i.
\end{equation}
In case
\begin{equation}\label{E:minidemp}
1_{\cdot}/|X|=\gX_1, 1_{\times}=\Psi_1\text{ are minimal idempotents,}
\end{equation} then
\begin{equation}\label{E:cacbrelspec}
\begin{split}
&A_i=|X|C_{i1},\quad B_j=C_{1j}\\
&\tr A=|X|\sum_iC_{i1},\quad \tr B=\sum_iC_{1j}.
\end{split}
\end{equation}
\end{enumerate}
\end{proposition}
\begin{proof}

Algebras $(Q,\cdot)$ and $(Q,\times)$ are semisimple. Indeed both $\pi,\mu$ permute maximal ideals $\{\mathfrak{m}\}$ of $(Q,\cdot)$ and act on their intersection $\mathfrak{rad}_{\cdot}=\bigcap \mathfrak{m}$. They also preserve filtration $\mathfrak{rad}_{\cdot}\supset \mathfrak{rad}_{\cdot}^2\supset\cdots $. Let $\mathfrak{rad}_{\cdot}^k$ be the last nonzero term of this filtration (we use that $\mathfrak{rad}_{\cdot}$ is nilpotent). Then $(a,a)=\tr_{\cdot}(a\cdot \pi a)=0\Rightarrow \mathfrak{rad}_{\cdot}=\{0\}$. Thus $(Q,\cdot)$ is a direct sum of fields. $\pi$ permutes minimal idempotents $\{\gX_i|i\in \setA\}$. We have $0<\tr_{\cdot}(\gX_i\cdot \pi \gX_i)$. From this we conclude that $\pi$ acts trivially on $\{\gX_i\}$.

We can prove the same way that $(Q,\times)$ is semisimple and $\mu$ acts trivially on the set of minimal idempotents $\{\Psi_i|i\in \setB\}$.
We have $\tr_{\cdot}(\gX_i\cdot \pi\gX_j)=\tr_{\cdot}(\gX_i\cdot \gX_j)=\tr_{\cdot}(\gX_i)\delta_{ij}=\delta_{ij}A_i$. From positivity of $(\cdot,\cdot)$ we conclude that $\{A_i\}$ are positive reals. The proof that $\mu\Psi_i=\Psi_i$ and $B_i>0$ is similar.

Operator $\mu$ defines an involution on $\setA$: $\gX_{\mu (i)}:=\mu\gX_i$. The same way $\pi$ defines an involution on $\setB$: $\Psi_{\pi (i)}:=\pi\Psi_i$. By abuse of notation we denote the involution of the set of idempotents by the symbol of the automorphism it induces.

Let us expand $\Psi_i$ in the basis of $\gX_i$:
\begin{equation}\label{E:psidef12}
\Psi_j=\sum_{s=1}^k d_{sj}\gX_s.
\end{equation}
Then 
\begin{equation}\label{E:DCmatdef}
C_{ij}=(\gX_i,\Psi_j)=\overline{d}_{ij}A_{i}.
\end{equation}
Equations (\ref{E:Bdef}) and (\ref{E:psidef12}) imply that 
\[B_{i}\delta_{ij}=(\Psi_i,{\Psi}_j)=
\sum_{s=1}^kd_{si}\overline{d}_{sj}A_s
=\sum_{s=1}^kA^{-1}_s\overline{C}_{si}C_{sj}.
\]
It is equivalent to the statement that $(C^{-1})_{ij}=B^{-1}_i\overline{C}_{ji}A^{-1}_j$, which implies that
\begin{equation}\label{E:secondeq}
A_{i}\delta_{ij}=\sum_{s=1}^kB^{-1}_s C_{is}\overline{C}_{js}.
\end{equation}

We leave to the reader the proof of 
 (\ref{E:invact}) which uses $\pi\mu=\mu\pi$ and (\ref{E:realtrace}). 

The triple $(A,B,C)$ is sufficient to recover two algebra structures up to an isomorphism. For this purpose we choose the standard basis $\{\gX_i\}$ for the vector space $\C^{\dim Q}$. We declare it a basis of idempotents for $\cdot$ -multiplication, $\tr_{\cdot}(\gX_i):=A_i$, $\mu\gX_i:=\gX_{\mu (i)}$, $\pi\gX_i:=\gX_{i}$. $\pi$ and $\mu$ obviously commute. Positivity of $\tr_{\cdot}(a\cdot\pi b)$ follows from (\ref{E:Adef}).

 We define the basis $\{\Psi_i\}$ of idempotents for $\times$ multiplication by the formula (\ref{E:psidef12}), where the matrix $d$ is extracted from (\ref{E:DCmatdef}). We also define $\tr_{\times}$ by $\tr_{\times}(\Psi_i):=B_i$. The formula (\ref{E:dotcomp}) would follow from (\ref{E:secondeq}).

Equations (\ref{E:cacbrelgen}),(\ref{E:cacbrelspec})
 follow from
\[A_i=\tr_{\cdot}\gX_i\overset{(\ref{Def:256})}{=}|X|\tr_{\times}(\gX_i\times 1_{\cdot}/|X|)=|X|\left(\gX_i,\sum_jn_j\Psi_j\right)\]
and
\[B_j=\tr_{\times}\Psi_j\overset{(\ref{Def:256})}{=}\tr_{\cdot}(1_{\times}\cdot \Psi_j)=\left(\sum_im_i\gX_i,\Psi_j\right).\]

\end{proof}
\begin{remark}\label{E:tensorproduct}
{\rm
Tensor product in monoidal category $\mathcal{T}$ transforms into tensor product of matrices $C,C'\to C\otimes C'$ and direct product of involutions $\pi,\pi'\to \pi\times\pi'$, $\mu,\mu'\to \mu\times\mu'$. Recall (see Proposition \ref{P:bases}) that involutions in the context of matrix $C$ are acting on indices of $C$.
}
\end{remark}

\subsection{The matrix $C(H,G)$}\label{S:matrixorbits}
 When $(H,G)$ is a Gelfand pair, collection $(Q(H,G)=\C[X\times X]^G, \cdot,\times,\tr_{\cdot},\tr_{\times},\pi,\mu
)$ (with the structures defined in (\ref{Def:1},\ref{Def:2},\ref{Def:2cdottimes},\ref{E:congc},\ref{E:transp})) defines a septuple $E(H,G)$ and a triple $\calA(Q(H,G))$ from Proposition \ref{P:equiv}. Semi-simple multiplications define bases $\{\gX_i\}$ and $\{\Psi_i\}$ of minimal idempotents in $Q(H,G)$. This way we get a triple $(A,B,C)$ (\ref{E:Adef},\ref{E:Bdef},\ref{E:Cdef}) associated with $(H,G)$.

By Propositions \ref{P:equiv} and \ref{P:bases} the triple $\calA(H,G)$ contains as much information as the triple $(A,B,C)$ and involutions $\pi,\mu$. (\ref{E:invact}). We will see that (\ref{E:minidemp}) is satisfied and $C$ completely determines $A$ and $B$.
Our goal for now is to find explicit formulas for $(A,B,C)$ that come from a finite Gelfand pair $H\subset G$.

From now on to the end of the paper  we assume that $X=G/H$.
\begin{remark}\label{R:xdef}
{\rm
Let $O_i\subset X\times X$ be a $G$-orbit in $ X\times X$ so that 
\begin{equation}\label{E:xxdecomposition}
X\times X=\bigsqcup_i O_i.
\end{equation}
 $O_i$ determines an $H$ orbit $o_i\subset X$:
\[\{x_0\} \times o_i :=O_i\cap X\times \{x_0\},\quad H=St(x_0).\]

Construction of $\gX_i \in\C[X\times X]^G$ is obvious: $\gX_i$ is the characteristic functions of $O_i$. The Gram-Schmidt of $\{\gX_i\}$ is 
\begin{equation}\label{E:dataA}
(\gX_i,\gX_j)=|O_i|\delta_{ij}
=\delta_{ij}|X| |o_i| =\delta_{ij}A_i.
\end{equation}
}
\end{remark}
\begin{remark}\label{E:diagonalidemp}
{\rm
Note that $1_{\times}=\gX_1$ that corresponds to the diagonal $X\subset X\times X$ is a minimal $\cdot$-idempotent.
}
\end{remark}
\begin{proposition}\label{P:integr}
\begin{enumerate}
\item In the product of $\cdot$-idempotents $\gX_i\times \gX_j=\sum_{s=1}^ka_{ijs}\gX_s$, $\gX_i\in \C[X\times X]^G$
\begin{equation}\label{E:aconditions}
\text{ the coefficients $a_{ijs}$ are nonnegative integers. }
\end{equation}
\item The integrality condition (\ref{E:aconditions}) in the coordinate form becomes
\begin{equation}\label{E:integrality}
\sum_sC_{is}C_{js}B^{-2}_s\overline{C}_{ms}A^{-1}_m\in \Z^{\geq 0}.
\end{equation}
\end{enumerate}
\end{proposition}
\begin{proof}
Suppose $(x,y)\in O_s$, where $O_s$ is the orbit that defines $\gX_s$. Then $a_{ijs}$ is the number of $z\in X$ such that $(x,z)\in O_i, (z,y)\in O_j$.

We leave verification of the second statement to the reader.
\end{proof}

\begin{corollary}
We can use the triple $(A,B,C)$ to write transition matrices between bases $\{\gX_s\}$ and $\{\Psi_j\}$:
\begin{equation}\label{E:psidef1}
\begin{split}
&\Psi_j=\sum_{s=1}^k \overline{C}_{sj}/A_{s}\gX_s\\
&\gX_s=\sum_{j=1}^k C_{sj}/B_{j}\Psi_j.
\end{split}
\end{equation}
\end{corollary}

\subsection{Zonal spherical function}\label{S:zonal}

The basis $\{\gX_i\}$ was identified in Remark \ref{R:xdef}. The description of the basis $\{\Psi_j\}$ uses the concept of zonal spherical function, or spherical function for short.
\begin{definition}
Let $(T(g),V)$ be a unitary finite-dimensional representation of of a finite $G$. 
We associate a function on $G$ defined by the formula:
\begin{equation}\label{E:psidef}
\psi_{\theta}(g)=\frac{\dim T}{|X|}(T(g)\theta,\theta)
\end{equation}
with a unit $H$-invariant vector $\theta$.
We call $\psi_{\theta}(g)$ the spherical function associated with a subgroup $H$ and unit invariant vector $\theta$. 
\end{definition}
The function $\psi_{\theta}(g)$ satisfies $\psi_{\theta}(hgh')=\psi_{\theta}(g)$
and descends to a function on $H\backslash X$. 

We will be writing $Spec_G(\C[X])$ for collection of irreducible representations with 
no repetitions that appear in (\ref{E:induceddec}).
Each representation is equipped with a fixed positive-definite Hermitian inner product. 
Let $T_i$ be a representation from $Spec_G(\C[X])$.
By construction $\C[X]=Ind_H^{G} \C$. 
By Frobenius formula
\[\Hom_H(Res^{G}_H T_{i},\C)\cong \Hom_G(T_{i},Ind_H^{G} \C)=\Hom_G(T_{i},\C[X])=V_i\cong \C.\]
 We conclude that $H$-invariant vector $\theta_{i}$ in $T_{i}$ is unique up to a scaling factor. We normalize it $(\theta_{i},\theta_{i})=1$.
\begin{convention}
It is convenient to use the elements of the set $Spec_G(\C[X])$ as labels of functions $\psi(g)$. Thus $\psi_i(g), i\in Spec_G(\C[X])$ will stand for the spherical function $\psi_{\theta}(g)$, where $\theta$ is a normalized $H$-invariant vector in the representation $T$ which belongs to the isomorphism class $i$.
\end{convention}

We define a Hermitian inner product on $G$ by the formula $(f,\bar{f}')=\sum_{g\in G}f(g)\bar{f'}(g)$.
\begin{proposition}\label{P:orthogonalL}
Functions $\{\psi_{i}(g)\}$ are orthogonal with respect to the inner product in $\C[G]$ and define a basis in $\C[X]^H$.
\end{proposition}
\begin{proof}
We extend vector $\theta_{i}=\theta_{i,1}$ to an orthonormal basis $\{\theta_{i,1},\dots, \theta_{i,\dim(T_i)}\}$. The function $\psi_{i}(g)$ is the matrix coefficient $T_{i,1,1}(g)$ of the representation $T_i$.
The following identity is a simple corollary of the Schur orthogonality relation for matrix coefficients of finite groups:
\begin{equation}\label{E:orthogonality}
\sum_{g\in G}\psi_{i}(g)\bar{\psi}_{j}(g)=\frac{\dim T_i\dim T_j}{|X|^2}\sum_{g\in G}T_{i, 1,1}(g)\overline{T}_{j, 1,1}(g)=
\delta_{i,j}|H|\frac{\dim T_i}{|X|}.
\end{equation}
 By orthogonality $\psi_i$ are linearly independent. Their number coincides with cardinality $|Spec_G(\C[X])|$, which is equal to the number of summands in (\ref{E:induceddec}). By assumption $\dim(V_i(X))\leq 1$. Thus, by (\ref{E:simpleiso}) and Maschke isomorphism 
\begin{equation}\label{E:Maschke1}\Hom_G(\C[X],\C[X])\overset{(\ref{Def:25})}{\cong} \C[X\times X]^G\overset{(\ref{E:induceddec})}{\cong} \bigoplus_{i=1}^k \Hom_{\C}(V_i,V_i)\end{equation}
$|Spec_G(\C[X])|=\dim\C[X\times X]^G=\dim\C[X]^H$.
\end{proof}
\begin{corollary}\label{L:orthogonalL}
We interpret $H\times H$ invariant function 
\begin{equation}\label{E:Psidef1}
\Psi_{i}=\psi_{i}(g^{-1}g')
\end{equation} on $G\times G$ as a function on $G/H\times G/H=X\times X$. By construction $\Psi_{i}$ is invariant under diagonal $G$-action.
Orthogonality relation
\begin{equation}\label{E:orthogonalityproduct}
\begin{split}
&B_{ij}=\sum_{x,y\in X\times X}\Psi_{i}(x,y)\overline{\Psi}_{j}(x,y)=
\delta_{i,j}\dim T_{i}
\end{split}
\end{equation}
follows from (\ref{E:orthogonality}).
\end{corollary}
\begin{proposition}\label{L:idempotents}
The functions $\Psi_{i}$
are minimal idempotents in 
$(\C[X\times X],\times)$.
The function $\Psi_{1}$ corresponding to trivial representation is equal to $1_{\cdot}/|X|$.
\end{proposition}
\begin{proof}
Schur orthogonality relations yield the proof:
\[
\begin{split}
&\sum_{y\in X}\Psi_{i}(x,y)\Psi_{j}(y,z)=
\frac{1}{|H|}\sum_{g'\in G}\psi_{i}(g^{-1}g')\psi_{j}({g'}^{-1}g'')=\\
&\frac{1}{|H|}\frac{\dim T_i\dim T_j}{|X|^2}\sum_{g'\in G}T_{i, 1,s}(g^{-1})T_{i, s,1}(g')T_{j, 1,t}({g'}^{-1})T_{j,t,1}(g'')=\\
&\frac{1}{|H|}\frac{\dim T_i\dim T_j}{|X|^2}\sum_{g'\in G}T_{i, 1,s}(g^{-1})T_{i, s,1}(g')\overline{T}_{j, t,1}(g')T_{j,t,1}(g'')=\\
&\delta_{i,j}\frac{1}{|H|}\frac{\dim T_i^2}{|X|^2}\frac{|G|}{\dim T_j}T_{i, 1,1}(g^{-1}g'')=
\delta_{i,j}\psi_{i}(g^{-1}g'')=\delta_{i,j}\Psi_{i}(x,z),\quad x=gH, z=g''H
\end{split}\]
\end{proof}
\subsection{Computing matrix $C$}\label{S:matrixC}
Having obtained explicit formulas for $\gX_i, \Psi_j$ we can compute matrix $C$ (matrices $A$ and $B$ were computed in (\ref{E:dataA}) and (\ref{E:orthogonalityproduct})).

\begin{proposition}
In each $H$-orbit $o_i\subset X$ choose $a_i\in o_i$. The indexing is chosen so that $St(a_1)=H$. We also fix $g_i\in G$ such that $a_i=g_ia_1$. There is a bijection between $\{g_i\}$ and $\{\gX_i\}$.
In the notations of isomorphism (\ref{E:induceddec})

\begin{equation}
C_{ij}=(\gX_i,\overline{\Psi}_j)
=|o_i|\dim T_j(\theta_j, T_j(g_i)\theta_j)
\end{equation}
$\theta_j$ is a normalized $H$-invariant vector in $T_j$.
\end{proposition}
\begin{proof}

\begin{equation}\label{E:ccomputation}
\begin{split}
&(\gX_i,\overline{\Psi}_j)=\sum_{(x,x')\in O_i}\overline{\Psi}_j(x,x')=\\
&\frac{1}{|H|^2}\sum_{g^{-1}g'\in Hg_iH}\overline{\psi}_j(g^{-1}g')=\frac{|G|}{|H|^2}\sum_{g\in Hg_iH}\overline{\psi}_j(g)=\\
&\frac{|G|}{|H|^2}\sum_{g\in Hg_iH}\frac{\dim T_j}{|X|}\overline{(T_j(g)\theta_j,\theta_j)}=|o_i|\dim T_j(\theta_j,T_j(g_i)\theta_j)\\
\end{split}
\end{equation}
\end{proof}

\begin{proposition}
Matrix $C$ satisfies
\begin{equation}\label{E:cfirsteq}
\sum_iC_{ij}=\delta_{1j}|X|,
\end{equation}
\begin{equation}\label{E:csecondeq1}
 \sum_jC_{ij}=\delta_{i1}|X| 
\end{equation}

\end{proposition}
\begin{proof}
The first equality is the corollary of orthogonality of $\Psi_1(x,x')=\frac{1}{|X|}$ and $\{\Psi_i(x,x')|i\neq 1\}$.
Indeed, by taking $O_i$ as in (\ref{E:xxdecomposition}) and using the first line of (\ref{E:ccomputation}) we get 
\[\delta_{1,j}=\frac{1}{|X|}\sum_{x,x'\in X} \overline{\Psi}_j(x,x') =\frac{1}{|X|}\sum_i\sum_{(x,x')\in O_i} \overline{\Psi}_j(x,x') =\frac{1}{|X|}\sum_{i}C_{i,j}.\]
This is equivalent to (\ref{E:cfirsteq}).

The second equality follows from Fourier expansion of the characteristic function $\Delta_X$ of the diagonal $X\subset X\times X$ :
\[\begin{split}
&\Delta_X(x,x')=\sum_j c_j \Psi_j(x,x'),\quad c_j= \frac{1}{\dim T_j}\sum_{x,x'\in X} \Delta_X(x,x')\overline{\Psi}_j(x,x')=\\
&=\frac{|X|}{\dim T_j}\overline{\psi}_j(1)=1.
\end{split}\]
So $\sum_j\Psi_j(x,x')=\Delta_X(x,x')$. In terms of the functions $\psi_j$ and characteristic function $\delta_{H}(g)$ of the subgroup $H$ this equality becomes
\[\sum_j\dim T_j(T_j(g)\theta_j,\theta_j)=|X|\delta_{H}(g).\]
On substituting $g\to g_i$ and multiplying both sides by $|o_i|$ this equality becomes equation (\ref{E:csecondeq1}).

\end{proof}

\begin{corollary}\label{C:corsumrowcol}
Matrix $C$ completely determines diagonal matrices $A$ and $B$
\begin{equation}\label{E:ABfromC}
\sum_iC_{1i}=\sum_iC_{i1}=|X|\quad A_{i}/|X|=C_{i1},\quad B_{i}=C_{1i}
\end{equation}
This follows from the inspection of (\ref{E:ccomputation}) or from Proposition \ref{P:bases} item (\ref{E:idempones}), Remark \ref{E:diagonalidemp}, and Proposition \ref{L:idempotents}.
\end{corollary}
\begin{remark}\label{R:character}
{\rm
Matrix $C$ for $(G,G\times G)$ ($G$ is embedded diagonally) is nothing else but the character table for the group $G$.
}
\end{remark}
\subsection{Arithmetic properties of $C$}\label{S:Arithmetic}
In this section we will take a closer look at the coefficients of the matrix $C$, which was derived from a Gelfand pair $(H,G)$.

\begin{proposition}
The coefficients of $C$ generate a Galois extension $L$ of $\Q$ whose Galois group is abelian. $L$ is a minimal splitting field for regular representation $G$ in $\Q[G/H]$.

\end{proposition}
\begin{proof}
Recall that exponent $e(G)$ of a finite group $G$ is a minimal positive number such that $g^{e(G)}=1 \forall g\in G$.
By \cite{CurtisReiner} Theorem 41.1 p. 292 any irreducible representation $T$ of a finite group $G$ defined over $K=\Q[\sqrt[k]{1}], k=e(G)$ has $\Hom_G(T,T)=K$.
The convolution algebra $K[X,X]^G$, being an algebra of endomorphisms of $K[X]$, by the above theorem is a direct sum $\prod_{i=1}^n K$ (see (\ref{E:Maschke1}) with $\C$ replaced by $K$). Basis $\{\gX_i \}$ belongs to $\Q[X,X]^G\subset K[X,X]^G$.
By the cited theorem idempotents $\{\Psi_i\}$ are defined over $K$ and the inner products $(\gX_i,\overline{\Psi}_j)\in K$. Let $L$ be a subfield of $K$ generated by $(\gX_i,\overline{\Psi}_j)$. The Galois group of $K$ is an abelian group $\Z_{k}^{\times}$. By the elementary Galois theory there is a subgroup $H(L)\subset \Z_{k}^{\times}$, which fixes elements of $L$ and $Gal(L/\Q)=\Z_{k}^{\times}/H(L)$.

$\Q[G/H]$ splits over $L$ because of the formula (\ref{E:psidef1}) for splitting idempotents $\Psi_j$. 

\end{proof}
\begin{proposition}
The entries of the matrix $C_{sj}/B_j$ are algebraic integers.
\end{proposition}
\begin{proof}
Let us expand $\gX_s$ in the basis $\{\Psi_j\}$ as in (\ref{E:psidef1}):
By definition $\Psi_j$ is an idempotent for $\times$-product. From this we deduce that
\[\gX_s\times \Psi_j=C_{sj}/B_j\Psi_j\]
It means that $\{\Psi_j\}$ is an eigenbasis for the operator $L_sy:=\gX_s\times y$ and $C_{sj}/B_j$ are its eigenvalues. The matrix of $L_s$ in the basis $\{\gX_s\}$ has integer coefficients (Proposition \ref{P:integr}). The characteristic polynomial of this matrix is monic with integer coefficients. By definition its roots are algebraic integers.
\end{proof}
\begin{remark}
{\rm
By Proposition 10.2 Section I in \cite{Neukirch} p.60 the set of algebraic integers in $\Q[\sqrt[k]{1}]$ ($k=e(G)$) is a free module $\Z[\zeta]$ over $\Z$. As a ring the module $\Z[\zeta]$ is generated by a primitive root of unity $\zeta$. Let $W_k\subset \Z[\zeta]$ be the group of roots of unity. $W_k$ contains a subset $A_k$ such that $|A_k|=|\Z_{k}^{\times}|$ (the orbit of $\zeta$ under the action of Galois group $\Z_{k}^{\times}$) such that $A_k$ is a basis for $\Z[\zeta]$.
 By utilizing Corollary \ref{C:corsumrowcol} we get
\begin{equation}\label{E:integrality2bis}
C_{sj}/C_{1j}=\sum_{\zeta^r\in A_k} m_{sjr}\zeta^r, m_{sjr}\in \Z.
\end{equation}
}
\end{remark}
\subsection{Duality isomorphism}\label{S:Dualityisomorphism}
Let us spell out explicitly how duality $D$ (\ref{E:Satakegen}) isomorphism is defined.
Let $(H,G)$ be a Gelfand pair and $(\cH,\cG)$ be the dual Gelfand pair. 

We choose a basis of minimal idempotents $\{\gX_i\}$ for $U
$ and $\{\Psi_j\}$ for $W
$. There are similar bases $\{\check{\gX}_i\}$ for $\check{U}$ and $\{\check{\Psi}_j\}$ for $\check{W}$.
The duality map $D$ acts by
\[D(\gX_i)=\check{\Psi}_{\sigma(i)},\quad D(\Psi_j)=\check{\gX}_{\tau(j)}\]
$\sigma$ and $\tau$ are some bijection between indexing sets.
The fact that $D$ is an isomorphism of structures (see Remark \ref{R:catTdef}) manifests itself in equalities 
\begin{equation}\label{E:identity1}
 C_{ij}=\overline{\check{C}}_{\tau(j)\sigma(i)}.
\end{equation}
Now it will be convenient to identify the set of labels $\setA,\setB$ with the set $\{1,\dots,|X|\}$ in a such a way that
$A_{i}\leq A_{i+1},B_{i}\leq B_{i+1}$. 
Without loss of generality we can assume that $\tau=\sigma=\id$. It follows from Corollary \ref{C:corsumrowcol} that 
\begin{equation}\label{E:identity0}
\begin{split}
&|X|=\sum_iC_{1i}=\sum_iC_{i1}=\sum_i\check{C}_{i1}=\sum_i\check{C}_{1i}=|\check{X}|\\
&\text{and } A_{i}/|X|=C_{i1}=\check{C}_{1i}=\check{B}_{i},\quad B_{i}=C_{1i}=\check{C}_{i1}=\check{A}_{i}/|\check{X}|.
\end{split}
\end{equation}

Therefore
\[|G/H|=|X|=|\check{X}|=|\check{G}/\check{H}|.\]
The same way from (\ref{E:dataA},\ref{E:orthogonalityproduct},\ref{E:identity0}) we read that 
\begin{equation}\label{E:orbitrep}
|o_i|=
\dim \check{T}_{i}.
\end{equation}
It means that duality interchanges sizes of $H$-orbits $a=A/|X|$ and dimensions $b=B$ of irreducible components the regular $G$ representation in $\C[X]$. We have seen this in the example \ref{P:illustration}.

Duality transformation produces from the triple $(A,B,C)$ a new triple 
\begin{equation}
(\check{A}:=|X|B,\check{B}:=|X|^{-1}A,\check{C}:=C^{\dagger})\text{ where } |X|:=\sum_i B_i
\end{equation}
with $\check{\pi}=\mu,\check{\mu}=\pi$ as in (\ref{E:invact}).
In order for the triple $(\check{A},\check{B},\check{C})$ to arise from a Gelfand pair $(\check{H},\check{G})$ it must satisfy some consistency conditions.
\begin{proposition}
The triple $(\check{A},\check{B},\check{C})$ satisfies (\ref{E:compatibility}),(\ref{E:cfirsteq}),(\ref{E:csecondeq1}),(\ref{E:ABfromC}) with all the variables replaced by their checked modifications.
\end{proposition}
\begin{proof}
Equation (\ref{E:compatibility}) tells us that $A^{-1}_{ii}\overline{C}_{ij}B^{-1}_{jj}$ is the inverse to the transpose of $C_{ij}$. From this we infer that $\sum_{k}C_{ik} A^{-1}_{jj}\overline{C}_{jk}B^{-1}_{kk}=\delta_{ij}$, which is equivalent to
\[\sum_{k}C_{ik} \overline{C}_{jk}\frac{1}{|X|B_{kk}}=\frac{1}{|X|}A_{jj} \delta_{ij}.\]

Duality interchanges equations (\ref{E:cfirsteq}),(\ref{E:csecondeq1}). It also swaps last pair of equations in (\ref{E:ABfromC}).

\end{proof}
\begin{definition}
 $(A,B,C)$ is called an integral triple if $(\check{A},\check{B},\check{C})$ satisfies (\ref{E:integrality}) and (\ref{E:integrality2bis}).
In terms of $(A,B,C)$ this conditions reads

\begin{equation}\label{E:integrality2}
|X|\sum_sC_{si}C_{sj}A^{-2}_s\overline{C}_{sm}B^{-1}_m\in \Z^{\geq 0}\quad \forall i,j,m
\end{equation}
\begin{equation}\label{E:integrality4}
C_{is}/C_{i1}=\sum_{\zeta^r\in A_k} n_{isr}\zeta^r, n_{isr}\in \Z, k=e(G)\quad \forall i,s.
\end{equation}
Strictly speaking the second equation is a corollary of the first. Still, it is worthwhile to write it separately because it is easy to verify.
\end{definition}

\begin{definition}\label{D:selfduality}

$(A,B,C)$ is called a self-dual triple if $A_{i}=B_{i}$ and there are $\sigma,\tau$ such that $A_{i}=A_{\sigma(i)}$, $B_{i}=B_{\tau(i)}$ and $\overline{C}_{ji}=C_{\sigma(i)\tau(j)}$. Moreover $\sigma \pi =\mu \sigma $, $\tau \pi =\mu \tau $.
\end{definition}

\begin{definition}\label{D:symmetry}
Let $L$ be a minimal splitting extension of $\Q$ for regular representation $\Q[X]$.
An element $(\sigma,\tau,g)\in S_{|X|}\times S_{|X|}\times Gal(L/\Q)$ is a symmetry of the triple
$(A,B,C)$ if $A_{i}=A_{\sigma(i)}$, $B_{i}=B_{\tau(i)}$, $C_{ij}=gC_{\sigma(i)\tau(j)}$. $\sigma \pi =\pi \sigma $, $\tau\mu=\mu\tau$.
As $L\subset \Q^{ab}\subset \C$, $g$ commutes with the complex conjugation.
\end{definition}

\subsection{Galois group of $C$ and duality}\label{S:Galoisgroupandduality}
 Let $L$ be the splitting field of the regular $\check{G}$-representation $\Q[\check{X}]$. An element $g$ of the Galois group $Gal(L/\Q)$ (which we know is abelian) permutes irreducible sub-representations of $L[\check{X}]$. This way $g$ defines a permutation $\check{\tau}_g$ of idempotents $\{\Psi_i\}$. As a result, we have an element $(1,\check{\tau}_{g^{-1}},g)$ of the group $ Aut(\check{A},\check{B},\check{C})$ defined in (\ref{D:symmetry}). By abuse of notation, we will say that $(1,\check{\tau}_{g},g)$ is an element of $Gal(L/\Q)$. The coefficients of $C$ and $C^{\dagger}$ generate the same subfield of $\C$. From this we conclude that $(\check{\tau}_{g},1,g)$ is a symmetry of $({A},{B},{C})$.

Let us define groups $Aut_{H}(G):=\{\phi\in Aut(G)|\phi H\subset H\}$ and $Out_{H}(G):=Aut_{H}(G)/H$. An element $k$ of $Aut_{H}(G)$ defines a permutation $\alpha_k$ of double cosets $H\backslash G/H$. Homomorphism $k\to \alpha_k$ factors through $Out_{H}(G)$. We interpret $\alpha_k$ as a permutation of idempotents $\{\gX_i\}$.
 $Aut_{H}(G)$ defines an automorphism of the group algebra $L[G]$. It leaves invariant subalgebra $\C[H\backslash G/H]$. This way we get a permutation $\beta_k$ of idempotents $\{\Psi_i\}$. To summarize: we have a homomorphism $\nu:Out_{H}(G)\to S_{|X|}\times S_{|X|}\times Gal(L/\Q)$,  
$\nu(k)=(\alpha_k,\beta_k,1)$, which is a symmetry of $(A,B,C)$ related to $H,G$. We will refer to elements $\{(\alpha_k,\beta_k,1)|k\in Z(Out_{H}(G))\}$ as elements of $Z_{H}(G)$.

In the above setup we conjecture that $(\check{\tau}_{g},\tau_{g^{-1}},1)$ is an isomorphism $Gal(L/\Q)$ with $Z_{H}(G)$.
$Gal(L/\Q)$ is the same for $H,G$ and $\check{H},\check{G}$. Conjecture implies that \[Z_{\check{H}}(\check{G})\cong Z_{H}(G).\]
 Suppose that $Aut(G)=G/Z(G)$. In this case $Out_{H}(G)=N(H)/H$. The group $N(H)$ is the normalizer of $H$ in $G$.
 There is an embedding $\C[N(H)/H]\to \C[H\backslash G/H]$. So if $\C[H\backslash G/H]$ is $\times$-commutative then $N(H)/H$ must be commutative. 
 Under above assumptions we conjecture that $Gal(L/\Q)\cong N(H)/H$.
\subsection{Maps between Gelfand pair}\label{S:MapsbetweenGelfandpair}
A map $(H,G)\to (H', G')$ of Gelfand pairs is a homomorphism $\phi:G\to G'$ that maps $H$ to $H'$.
A map $\phi$ is a strong Hecke equivalence if it induces a bijection $\phi:G/H\to G'/H'$. We say that $\phi$ is a Hecke equivalence if it induces a bijection $H\backslash G/H\to H'\backslash G'/H'$. 
\begin{proposition}
Hecke equivalence $\phi:(H, G)\to (H', G')$ of finite Gelfand pairs implies strong Hecke equivalence.
\end{proposition}
\begin{proof}
$\phi:G/H=:X\to X':=G'/H'$ induces a map of representations $\phi^*:\C[X']\to \C[X]$ via pullback.
By the assumption it defines an isomorphism of invariants 
\begin{equation}\label{E:weak}
\C[X']^{H'}\to \C[X]^{H}.
\end{equation}
By Proposition \ref{P:orthogonalL} spherical functions $\{\psi'_i\}$ 
 form a basis in $\C[X']^{H'}$. In addition the span of left shifts $\psi'_i(gx), x\in X'$ of $\psi'_i(x)$ coincides with the irreducible representation $T_i'$. The similar construction applied to $\psi'_i(\phi(x))$ generates an isomorphic subrepresentation in $\C[X]$. Representations $T_i'$ come with multiplicity one and generate $\C[X']$.
 From this we conclude that $\phi^*$ is injective. From isomorphism (\ref{E:weak}) we see that $\psi_j(x)=\sum a_{ij} \psi'_i(\phi(x))$ where $(a_{ij})$ is an invertible matrix. We know that spherical functions are orthogonal. This explains why $\psi_i(x)=\psi'_j(\phi(x))$ for some $j_i$.
\end{proof}
\begin{example}
Here is a diagram, created with a help of a computer, of Hecke equivalences between all the Gelfand pairs that produce matrices $(A,B,C)$ (\ref{E:SN}) with $n=7$:
 \[(S_4,PSL(3,\Z_2))\to (A_6,A_7)\to (S_6, S_7) \leftarrow (\Z_3,(\Z_7 \rtimes \Z_3) \rtimes \Z_2).\] 
Note that the graph is connected and all pairs are (strongly) Hecke equivalent to $(S_6, S_7)$.
\end{example}

\begin{example}
We used a computer to generate the diagram 
 of all Hecke equivalences of Gelfand pairs that correspond to matrix $C$ 
 
\[
\left(
\begin{array}{rrr}
 1 & 1 & 6 \\
 3 & 3 & -6 \\
 4 & -4 & 0 \\
\end{array}
\right).
\]
The matrix is a member of a bigger family matrixes related to $S_{n} \wr S_{k}$ ($n=4,k=2$)and discussed in Section \ref{S:wr}.
 We have a connected diagram of equivalences with the terminal group equal to $S_{4} \wr S_{2}$ (we omit precise formulas for inclusion $H\subset G$):
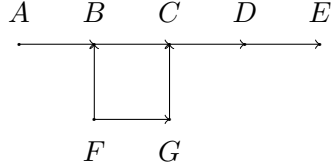
\begin{figure}[h]
\centering
 \begin{tikzpicture}[scale=1.0]
	\begin{pgfonlayer}{nodelayer}
		\node [label=90:$D$] (0) at (0, 3) {.};
		\node [label=90:$E$] (1) at (1, 3) {.};
		\node [label=90:$C$] (2) at (-1, 3) {.};
		\node [label=-90:$G$] (3) at (-1, 2) {.};
		\node [label=90:$B$] (4) at (-2, 3) {.};
		\node [label=90:$A$] (5) at (-3, 3) {.};
		\node [label=-90:$F$] (6) at (-2, 2) {.};
	\end{pgfonlayer}
	\begin{pgfonlayer}{edgelayer}
		\draw [->](5.center) to (4.center);
		\draw [->](4.center) to (2.center);
		\draw [->](2.center) to (0.center);
		\draw [->](0.center) to (1.center);
		\draw [->](6.center) to (4.center);
		\draw [->](3.center) to (2.center);
		\draw [->](6.center) to (3.center);
	\end{pgfonlayer}
\end{tikzpicture}
\end{figure}
\[\begin{split}
&A=A_4\subset \left(\left(\Z_2^{\times 4}\right) \rtimes \Z_3\right) \rtimes \Z_2\\
&B=S_4\subset (\left(\left(\Z_2^{\times 4}\right) \rtimes \Z_3\right) \rtimes \Z_2)\rtimes \Z_2\\
&C=(\Z_3 \times A_4) \rtimes \Z_2\subset (A_4^{\times 2} \rtimes \Z_2) \rtimes \Z_2\\
&D=(\Z_3 \times A_4) \rtimes \Z_2\subset (A_4^{\times 2}) \rtimes \Z_4\\
&E=S_4^{\times 2}\subset (S_4^{\times 2}) \rtimes\Z_2\\
&F=A_4\subset \left(\left(\Z_2^{\times 4}\right) \rtimes \Z_3\right) \rtimes \Z_2\\
&G=\Z_3\times A_4\subset (A_4^{\times 2}) \rtimes\Z_2\\
\end{split}\]
Here is the diagram of maps of the dual family of groups. The pair which corresponds to $S_{2} \wr S_{4}$ from Section \ref{S:wr} is $C'$:
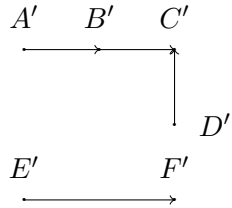
\begin{figure}[h]
\centering
\begin{tikzpicture}
	\begin{pgfonlayer}{nodelayer}
		\node [label=90:$A'$] (0) at (-4, 3) {.};
		\node [label=90:$B'$] (4) at (-3, 3) {.};
		\node [label=90:$C'$] (3) at (-2, 3) {.};
		\node [label=0:$D'$] (2) at (-2, 2) {.};
		\node [label=90:$E'$] (1) at (-4, 1) {.};
		\node [label=90:$F'$] (5) at (-2, 1) {.};
	\end{pgfonlayer}
	\begin{pgfonlayer}{edgelayer}
		\draw [->] (4.center) to (3.center);
		\draw [->] (2.center) to (3.center);
		\draw [->] (0.center) to (4.center);
		\draw [->] (1.center) to (5.center);
	\end{pgfonlayer}
\end{tikzpicture}
\end{figure}
\[\begin{split}
&A'=S_3\subset \GL(2,\Z_3)\\
&B'=S_4\subset (((\Z_2 \times D_8) \rtimes \Z_2) \rtimes \Z_3) \rtimes \Z_2\\
&C'=\Z_2 \times S_4\subset \Z_2^4\rtimes S_4\\
&D'=\Z_2 \times A_4\subset \Z_2^4\rtimes A_4\\
&E'=A_4\subset ((\Z_2 \times D_8) \rtimes \Z_2) \rtimes \Z_3 \\
&F'=S_4 \subset (((\Z_2 \times D_8) \rtimes \Z_2) \rtimes \Z_3) \rtimes \Z_2 
\end{split}\]
We see that the diagram of Hecke equivalences for the dual family is disconnected. 
\end{example}
We come to the conclusion that the dual pair $(\check{H},\check{G})$ for $({H},{G})$, even it exists, is not necessarily unique even if study pairs up to Hecke equivalence.
\section{Examples}\label{S:Examples}
In this section we discuss some examples of the proposed duality. We will formulate a classical sufficient condition for $(H,G)$ to be a Gelfand pair.
\begin{proposition}\label{G:gelfcrit}
Let $\psi:G\to G$ be an involutive anti-automorphism such that $\psi$ leaves each double cosets $HgH$ invariant. Then $(H,G)$ is a Gelfand pair.
\end{proposition}
\subsection{$(S_{n-1}, S_n)$}\label{S:SN}
Let $S_n$ be a symmetric group on $n$ letters ($n\geq 2, S_{1}:=\{1\}$). The map $\sigma\to \sigma^{-1}$ leaves $S_{n-1}$ and $S_{n-1}\sigma_{n-1,n}S_{n-1}$ invariant. So, by Proposition \ref{G:gelfcrit} $(S_{n-1}, S_n)$ is a Gelfand pair. The set $X=S_n/S_{n-1}$ is isomorphic to $\{1,\dots,n\}$ with the natural $S_{n}$-action. $X$ is a union of two $S_{n-1}$-orbits $\{1,\dots,n-1\}\cup \{n\}$. The regular $S_n$ representation in $\C[X]$ is a direct sum of the trivial representation $\C$ and $n-1$-dimensional defining representation. 
We conclude that $A=n(1,n-1),B=(1,n-1)$ and $C$ is
\begin{equation}\label{E:SN}
L_n:=\left(
\begin{array}{rr}
1 & n-1 \\
 n-1 & -(n-1) \\
\end{array}
\right).
\end{equation}
The formulas for $A$ and $B$ follow from equations (\ref{E:dataA}) and (\ref{E:orthogonalityproduct}). The formulas for $A$ and $B$ do not support nontrivial involutions $\pi$ and $\mu$. The formulas for $C$ follows from Proposition \ref{E:cfirsteq} and Corollary \ref{C:corsumrowcol}.

The symmetry of the matrix $C$ implies that $(S_{n-1},S_n)$ is a self-dual pair.

\subsection{The group $S_{n} \wr S_{k}$}\label{S:wr}
A more complex example is the wreath product $G=S_{n} \wr S_{k}=S_{n}^{\times k}\rtimes S_k, n,k\geq 2$. It is a symmetry group of a set of pairs \[X=\{(i,j)| 1\leq i\leq n, 1\leq j\leq k\}.\] An element $g=(\sigma_1,\dots,\sigma_k)\rtimes \sigma\in S_{n}^{\times k}\rtimes S_k$ acts by the formula
\[g(i,j)=(\sigma_{l}(i), \sigma (j)), l=\sigma (j).\]
The stabilizer of $(1,1)$ is the subgroup $H=(S_{n-1}\times S_{n}^{\times (k-1)})\rtimes S_{k-1}$. 

The $H$-orbits of $H$ in $X$ are \[o_1\cup o_2\cup o_3:=\{(1,1)\}\cup \{(i,1)|2\leq i\leq n\}\cup \{(i,j)|1\leq i\leq n, 2\leq j\leq k\}.\] Their respective cardinalities are $1,n-1, n(k-1)$.

From this we conclude that $G=H\cup Hg_1H\cup Hg_2H$ where
\[g_1=(\sigma_{1,2}\times 1^{\times (k-1)})\rtimes 1,\quad g_2=(1^{\times (k)})\rtimes \sigma_{1,2}.\]
The anti-involution $g\to g^{-1}$ leaves all three double cosets invariant and by Proposition \ref{G:gelfcrit} we know that $(H,G)$ is a Gelfand pair.

The group $G$ is compatible with projection $p:X\to X_0=\{1\,\dots,k\}, p(i,j)=j$. $G$ action on the image of $p$ factors through $S_k$. The pullback along $p$ defines an inclusion $\C[X_0]\subset \C[X]$. $\C[X_0]$ is a direct sum of the trivial $\C$ and $k-1$ dimensional representation $T$. The orthogonal complement $Q$ to $\C[X_0]$ in $\C[X]$ has dimension $kn-(1+(k-1))=(n-1)k$. It must be an irreducible representation because $(H,G)$ is a Gelfand pair and $\dim \C[X\times X]^G=3$.

Let us compute function $\Psi_2$ corresponding to representation $T$. According to the formulas (\ref{E:psidef}) and (\ref{E:Psidef1})
this is done in terms of $H$-invariant normalized vector $\theta\in \left\{\sum_{i=1}^ka_ie_i|\sum_{i=1}^ka_i=0\right\}$. In this formula $\{e_i\}$ is the standard basis for $\C[X_0]$. We choose it to be
\[\theta=\frac{1}{\sqrt{k(k-1)}}\left((k-1)e_1-\sum_{i=2}^k e_i\right).\]
The function $(T(g)\theta,\theta), g\in S_k$ from (\ref{E:psidef}) has the following values:
\[(T(g)\theta,\theta)= 
\begin{cases}
1&\text{ if } g\in S_{k-1}\\
-\frac{1}{k-1} &\text{ if } g\in S_{k-1}\sigma_{12}S_{k-1}.
\end{cases}\]
From this we immediately compute the values of $\Psi_{2}(x,x_0)$, where $x_0=(1,1)$:
\[\Psi_{2}(x,x_0)= 
\begin{cases}
\frac{\dim T}{|X|}=\frac{k-1}{nk}&\text{ if } x=(i,1), 1\leq i\leq n\\
-\frac{\dim T}{|X|(k-1)}=-\frac{1}{nk} &\text{ if } x=(i,j), 1\leq i\leq n, 2\leq j\leq k.
\end{cases}\]
From this we derive the following formulas:
\[\begin{split}
&(\gX_i, \Psi_2)=|X|\sum_{x\in X}\gX_i(x,x_0)\overline{\Psi}_2(x,x_0)=\\
&=\begin{cases}
k-1&, i=1\\
(n-1)(k-1)&, i=2\\
-(k-1)n&, i=3.
\end{cases}
\end{split}\]
These numbers constitute the second column of the matrix $C$. We recover the third column from equation (\ref{E:csecondeq1}). Here is the final answer for $C$:

\begin{equation}\label{E:matrixM}
M_{n,k}=\left(
\begin{array}{rrr}
 1 & k-1 & k (n-1) \\
 n-1 & (k-1) (n-1) & -k (n-1) \\
 (k-1) n & (1-k) n & 0 \\
\end{array}
\right).
\end{equation}
As the double cosets are invariant with resect to $g\to g^{-1}$, we conclude that $\gX_i$ are invariant with respect to $\pi$ and $\mu$. $\pi$ and $\mu$ act trivially on $\Psi_i$ because all the irreducible sub-representations in $\C[X]$ are real. 

We see that $C^{\dagger}$ corresponds to the column-wise action of $S_{k} \wr S_{n}$ on the set $X$. It is remarkable that duality doesn't preserve cardinalities of the groups for $ |S_{n} \wr S_{k}|=n!^kk!\neq k!^nn!= |S_{k} \wr S_{n}|$.

\subsection{Semidirect product with an abelian group}\label{S:semidirect}
In this section we will find the dual pair for the pair $(H,A\rtimes H)$. In this pair groups $A$ and $H$ are finite. $A$ is an abelian group, equipped with an $H$-action. The embedding of $H$ is defined by $h\to (0,h)$. 

There is an anti-involution $\nu$ on $G$ defined by the formula $(a,h)\to (h^{-1}(a),h^{-1})$. It is the composition of the inverse and the involution $(a,h)\to (-a,h)$. $\nu$ preserves double cosets $HaH$. By Proposition \ref{G:gelfcrit} $(H,A\rtimes H)$ is a Gelfand pair. Note that the homogeneous space $X$ coincides with $A$. The group $H$ is acting on $X=A$ according to homomorphism $\phi:H\to Aut(A)$, which defines the semi-direct product. The group $A$ acts on $X$ by translations.

We use adjoint to $\phi$ to define the semidirect product $\check{A}\rtimes H$, where $\check{A}$ is the Pontryagin dual $\Hom(A,\C^{\times})$.

In this section we show that $(H,\check{A}\rtimes H)$ is dual to $(H,A\rtimes H)$. Note that even though $\check{A}\rtimes H$ and $A\rtimes H$ have the same cardinalities in general the groups are not isomorphic. The simplest example is $H=\mathbb{F}_p^{\times}$ acting on $A=\mathbb{F}_p$-additive group of the finite field. 

The $H$ orbits in $X$ are the orbits of $H$-action on $A$. To simplify notations we fix a representative $a_i\in o_i\subset A$, $\chi_j\in o_j\subset \check{A}$ in each $H$-orbit.
This way 
\begin{equation}\label{E:union}
A=\bigsqcup_{i=1}^ko(a_i),\quad \check{A}=\bigsqcup_{i=1}^ko(\chi_i).
\end{equation}

An $o(\chi_i)$-graded representation $T$ of the group $G$ is a representation that decomposes into the direct sum of linear spaces labelled by characters from a fixed $G$-orbit:
\[T_{o(\chi_i)}=\bigoplus_{\chi\in o(\chi_i)} T_{\chi}.\]
The group $G$ acts on $T$ by isomorphisms $g:T_{\chi}\to T_{g\chi}$. The groups $A$ acts on $T_{\chi}$ according to the character $\chi$.

\begin{proposition}\label{P:decomposition}
\begin{enumerate}
\item Any representation $T$ of $G$ has a decomposition 
\begin{equation}\label{E:decomposition}
T=\bigoplus_{i=1}^k T_{o(\chi_i)}.
\end{equation} 
\item $T_{o(\chi_i)}$ is $G$-irreducible iff the graded component ${T_{o(\chi_i)}}_{\chi_i}$ is an irreducible representation of $St(\chi_i)$.
\item $\C[X]$ splits into the direct sums (\ref{E:decomposition}) where $T_{o(\chi_i)}\neq \{0\}$ for each $i$. Moreover dimensions of $o(\chi_i)$-graded components $T_{\chi, o(\chi_i)}$ are equal to one.
\end{enumerate}
\end{proposition}
\begin{proof}
 $T$ decomposes into $A$-isotypic components by the characters $\chi\in \check{A}$. The group $G$ shuffles these components. The characters can be grouped according to orbits $o(\chi_i)$. By construction, $\bigoplus_{\chi\in o(\chi_i)} T_{\chi}$ is a subrepresentation in $T$. The proof follows from (\ref{E:union}).
 
 If $T'_{\chi_i}\subset T_{\chi_i}$ is a proper $St(\chi_i)$-subrepresentation, then $\sum_{g\in G}gT'_{\chi_i}$ is a $G$ subrepresentation of $T_{o(\chi_i)}$. Hence $T_{o(\chi_i)}$ is not irreducible. Conversely, any irreducible $G$- subrepresentation $P$ of $T_{o(\chi_i)}$ is generated by $P_{\chi_i}\subset T_{\chi_i}$.
 
 As a representation of the group $A$, $\C[X]$ is isomorphic to $ \C[A]$. From the Fourier analysis on $A$, we know that its $A$-isotypic components are one-dimensional. Hence $\dim \C[X]_{\chi}=1$ for all $\chi$. The second statement follows from this. 
\end{proof}

Our next goal is to find a formula for the functions $\Phi_i$. To do this we have to find a $G$-invariant vector $\theta\in T_{o(\chi_i)}$.

\begin{proposition}\label{P:thetadef}
\begin{equation}\label{E:thetadef}
 \theta(a)=\sum_{\chi\in o(\chi_i)} \chi(a)
 \end{equation}
 is an $H$-invariant function on $X$.
\end{proposition}
\begin{proof}
Let $ \theta(a)=\sum_{\chi\in o(\chi_i)} \theta_{\chi}\chi(a)\in T_{o(\chi_i)}\subset \C[X]$ be an $H$-invariant function. 
$\theta$ satisfies 
\[\theta((0,g^{-1})a)=\sum_{\chi\in o(\chi_i)} \theta_{\chi}\chi(g^{-1}a))=\sum_{\chi\in o(\chi_i)} \theta_{\chi}\chi(a)=\theta(a).\]

It implies that $\theta_{g^{-1}\chi}=\theta_{\chi}$.
\end{proof}

The function 
\[\theta_{o(\chi_i)}(a)=\frac{1}{\sqrt{|\check{A}||o(\chi_i)|}}\sum_{\chi\in o(\chi_i)} \chi(a)\]
represents an $H$-invariant unit vector in $T_{o(\chi_i)}$. After a little algebra we find that
\[(\theta_i,T_i(l,g)\theta_i)=\frac{1}{|o(\chi_i)|}\sum_{\chi\in O(\chi_i)} \chi(l),\quad \psi_{i}(l,g)=\frac{1}{|A|}\sum_{\chi\in o(\chi_i)}\chi(l)\]

The triple $(A,B,C)$ in our case is
\begin{equation}\label{E:dataX}
\begin{split}
&(\gX_i,{\gX}_j)
=\delta_{ij}|A| |o(a_i)|= \delta_{ij}A_i, \\
& (\Psi_i,{\Psi}_j)
=\delta_{ij}|o(\chi_i)|=\delta_{ij}B_i, \\
&(\gX_i,{\Psi}_j)
=\sum_{\substack{\chi\in o(\chi_j),\\a\in o(a_i)}} \chi^{-1}(a)=C_{ij}.
\end{split}
\end{equation}

Equations (\ref{E:dataX}) make it obvious that a possible dual pair to $H\subset A\rtimes H$ is $H\subset \check{A}\rtimes H$.

\subsection{Non examples}\label{S:Nonexample}
There are some classical examples of Gelfand pairs that do not have any dual counterparts. One set of such non examples is formed by $S_k\times S_{n-k}\subset S_{n}$ (See \cite{MWildon} for classification of multiplicity free subgroups in $S_n$.) Computations with GAP4 (we omit details) show that $C_n$ derived from $S_2\times S_{n-2}\subset S_{n}$ is
\[C_n=\left(
\begin{array}{rrr}
 1 & n+1 & \frac{1}{2} (n-1) (n+2) \\
 2 n & (n-2) (n+1) & -(n-1) (n+2) \\
 \frac{1}{2} (n-1) n & -(n-1) (n+1) & \frac{1}{2} (n-1) (n+2) \\
\end{array}
\right).\]
$C_n$
 fails integrality test (\ref{E:integrality4}) for $n\geq 3$. 
 
 Gelfand pairs $(G,G\times G)$ described in Remark \ref{R:character} almost never pass integrality tests (\ref{E:integrality2},\ref{E:integrality4}). Rare exceptions from the above rule come  some solvable $G$. The nature of these groups deserve a closer study.

\subsection{On classification of dual pairs}\label{S:Onclassification}
In this section we present our attempt to compile a list of matrices $C$ that come from dual Gelfand pairs. The computation of matrix $C$, if done by hand, could be very time-consuming. This is why we used the GAP4 system to accelerate the calculations. By (\ref{E:identity0}) duality doesn't change $|X|$. GAP4 has a library of all transitive effective actions on sets with cardinality $\leq 30$. Given a group and a transitive action on $X$ we recover $H$ without troubles, as a stabilizer of a point. This way we can systematically search for dual pairs. The data that must be present in an ideal list might take some space. 
In such a list each matrix $C$ should be accompanied by the description of the generators of $G\subset S_{|X|}$. The list should also contain a description of all Hecke equivalences $\phi:G\to G'$ in the case $C$ happen to correspond to more then one pair. It is clear that there is a room in the list for more interesting details. To save space I decided to drastically  limit myself with a description only of $C$ and with a record of cardinality $r(C)$ of the set $\{(H,G)|C(H,G)=C\}$ of the isomorphism classes of pairs. I will also omit matrices that break into nontrivial tensor products.

Let us fix the notations used in the table. Matrix $L_n$ is defined in (\ref{E:SN}). Matrix $E_n$ is $(\zeta_n^{ij}), i,j=0,\dots n-1$. $\zeta_n$ is a primitive root of unity of order $n$. Matrix $M_{n,k}$ is defined in (\ref{E:matrixM}). In the table below we will be writing $r(C)$ in the exponent of $C$. Double entries in the second column correspond to dual pairs. Single entries correspond to self-dual pairs:
\[\begin{array}{|c|c|}
\hline
|X| & C\quad Out_{H}(G),Gal(L/\Q) \\
\hline
2&\begin{array}{c}L_2^1\quad \{1\},\{1\}\end{array}\\
\hline
3&\begin{array}{c}L_3^1\quad \{1\},\{1\}\\ E_3^1\quad \Z_2,\Z_2\\ \end{array}\\
\hline
4&\begin{array}{c}L_4^2\quad \{1\},\{1\} \\ M_{2,2}^1\quad \{1\},\{1\}\\ E_4^1 \quad \Z_2,\Z_2 \\ \end{array}\\
\hline
5&\begin{array}{c}L_5^3\quad \{1\},\{1\}\\ \scalebox{0.5}{$ \left(
\begin{array}{rrr}
 1 & 2 & 2 \\
 2 & 2 \zeta_5 ^3+2 \zeta_5 ^2 & 2 \zeta_5 ^4+2 \zeta_5 \\
 2 & 2 \zeta_5 ^4+2 \zeta_5 & 2 \zeta_5 ^3+2 \zeta_5 ^2 \\
\end{array}
\right)$}^1\quad \Z_2,\Z_2
\\ E_5^1\quad \Z_4,\Z_4\\ \end{array}\\
\hline
6&\begin{array}{c}L_6^4\quad \{1\},\{1\}\\ 
\begin{array}{c|c}M_{2,3}^3\quad \{1\},\{1\} &M_{3,2}^3\quad \{1\},\{1\}\end{array}\\
\begin{array}{c|c}\scalebox{0.5}{$\left(
\begin{array}{rrrr}
 1 & 1 & 1 & 3 \\
 1 & 1 & 1 & -3 \\
 2 & 2 \zeta _3 & 2 \zeta _3^2 & 0 \\
 2 & 2 \zeta _3^2 & 2 \zeta _3 & 0 \\
\end{array}
\right)$}^2
\quad \Z_2,\Z_2&\scalebox{0.5}{$\left(
\begin{array}{rrrr}
 1 & 1 & 2 & 2 \\
 1 & 1 & 2 \zeta _3^2 & 2 \zeta _3 \\
 1 & 1 & 2 \zeta _3 & 2 \zeta _3^2 \\
 3 & -3 & 0 & 0 \\
\end{array}
\right)$}^1\quad \Z_2,\Z_2\end{array}
 \end{array}\\
\hline
7&\begin{array}{c}L_7^4\quad \{1\},\{1\}\\ 
\scalebox{0.5}{$ \left(
\begin{array}{rrr}
 1 & 3 & 3 \\
 3 & 3 \zeta _7^6+3 \zeta _7^5+3 \zeta _7^3 & 3 \zeta _7^4+3 \zeta _7^2+3 \zeta _7 \\
 3 & 3 \zeta _7^4+3 \zeta _7^2+3 \zeta _7 & 3 \zeta _7^6+3 \zeta _7^5+3 \zeta _7^3 \\
\end{array}
\right)$}^1*\quad \Z_2,\Z_2\\ 
\scalebox{0.5}{$ \left(
\begin{array}{rrrr}
 1 & 2 & 2 & 2 \\
 2 & 2 \zeta _7^4+2 \zeta _7^3 & 2 \zeta _7^5+2 \zeta _7^2 & 2 \zeta _7^6+2 \zeta _7 \\
 2 & 2 \zeta _7^6+2 \zeta _7 & 2 \zeta _7^4+2 \zeta _7^3 & 2 \zeta _7^5+2 \zeta _7^2 \\
 2 & 2 \zeta _7^5+2 \zeta _7^2 & 2 \zeta _7^6+2 \zeta _7 & 2 \zeta _7^4+2 \zeta _7^3 \\
\end{array}
\right)$}^1\quad \Z_3,\Z_3\\
E_7^1 \quad \Z_6,\Z_6
\\ \end{array}\\
\hline
\cdots&\cdots
\end{array}
\]

The matrix marked by astérisque $*$ corresponds to a self-dual pair with a nontrivial self-duality isomorphism (\ref{D:selfduality}).
The number of entries in the table grows fast with $|X|$ (for $|X|=12$ this number is about $30$) and proposed presentation of data becomes impractical.

In the adopted notation matrix $C$ corresponding to $(S_3,S_3\times S_3)$, from example (\ref{P:illustration}), is $M_{3,2}$. The matrix $M_{2,3}$ corresponds to $(\Z_4,S_4)$.
\subsection{Some sporadic examples}
Mathieu group $G=M_{11}$, the smallest in the sporadic family,  acts transitively on the $22$ objects (see \cite{Conway} p.221 for details). In this case $H$ is the alternating group $A_6$. The $C$ matrix is
\[\left(
\begin{array}{rrr}
 1 & 10 & 11 \\
 1 & 10 & -11 \\
 20 & -20 & 0 \\
\end{array}
\right).\]
The transposed matrix corresponds to the group $\check{G}=(M_{11} \times M_{11}) \rtimes \Z_2$ with the subgroup $\check{H}$ isomorphic to $M_{11} \times S_6$. The action of $\check{G}$ comes from the action of $M_{11}$ on $\Omega, |\Omega|=11$. Let us take two copies of $\Omega$. Then $M_{11} \times M_{11}$ acts on $\Omega\coprod \Omega'$. The generator $a\in \Z_2$ acts by $\omega_i\to \omega'_i$. Matrix $C$ doesn't support nontrivial  $\pi$ and $\mu$.

There is a self-dual example with $C$ equal to 
\[\left(
\begin{array}{rrrrr}
 1 & 11 & 11 & 55 & 66 \\
 11 & 121 & -11 & -55 & -66 \\
 11 & -11 & 121 & -55 & -66 \\
 55 & -55 & -55 & 385 & -330 \\
 66 & -66 & -66 & -330 & 396 \\
\end{array}
\right).\]
It is related to the exceptional group $G=M_{12}$ and a subgroup $H=\mathrm{PSL}_2(\Z_{11})\subset M_{11}\subset M_{12}$. Involutions $\pi$ and $\mu$ also act trivially. For more on multiplicity free subgroups in sporadic groups see \cite{BreuerLux}.
\section{Some open questions}\label{S:concluding}
In this section I collected several questions related to the subject of the paper that are worth of a closer look. 

\subsection{On the relation to other dualities}The first question that needs to be addressed is rather vague: what is the relation of the proposed duality to other dualities known in mathematics: Mirror symmetry and Langlands duality. Though, the constructions in this paper were partly motivated by Satake isomorphism, the relation to main body of Langlands conjectures is not clear. The spaces $L(N,M)$ (or better $L(N,N',M)=\left\{\gamma:[0,1]\to M|\gamma(0)\in N,\gamma(1)\in N'\right\}$) where $N,N'$ are Lagrangian submanifold in symplectic $M$, play an important role in the definition of Floer homology and Fukaya category. This might give an indication  of a possible relation of proposed duality to Mirror symmetry.
\subsection{Duality for Gelfand pairs of Lie groups and locally compact groups}
In order to define a Gelfand pair in the context of differential geometry we consider the algebra of differential operators $Diff(X)$ on the homogeneous space $X=G/H$ (see \cite{VinbergCom} for more details). The algebra $Diff(X)^G$ in general is noncommutative. When it is commutative $H,G$ is by definition a Gelfand pair. $Diff(X)^G$ is analogous to $(\C[X\times X]^G,\times)$. A possible candidate for the role of $(\C[X\times X]^G,\cdot)$ is the space of $G$-invariant functions on the cotangent bundle $T^*X$ with point-wise multiplication. It is not clear if $Diff(X)^G$ and $C^{\infty}(T^*X)$ support any traces, which are necessary to put them in our framework. Modulo these details one can ask a question of constructing a dual pair $\check{H},\check{G}$, where the role $Diff(X)^G$ and $C^{\infty}(T^*X)$ is interchanged. It is not clear at present if this duality can be illustrated by any interesting examples. The simplest question that waits its answer is what is the dual pair for $\mathrm{S}^1\subset \mathrm{SU}_2$ and whether  the dual pair exists at all. Perhaps, we have to extend the area of search to include locally compact dual pairs.

The fundamental group, which was used in Example \ref{Ex:orb}, is a part of a package that contains also the higher homotopy groups. The higher homotopy groups form a graded Lie algebra with respect to Whitehead product. In fact the theory of rational homotopy type of Sullivan tells us that the the homotopy category of rational differential graded Lie algebras with
nilpotent finite type homology is equivalent to the rational homotopy category of nilpotent topological spaces with finite type rational homology. Would it be possible to combine this equivalence and the Lie algebra technique used in the differential geometry to construct interesting examples of topological Gelfand pairs?

\subsection{On the categorification} Objects of the category $Vect[X\times X]^G$ \cite{EtingofGelakiNikshychOstrik} are finite-dimensional linear spaces $V_{x,y}, x,y\in X$ together with a finite group action. $R(g):V_{x,y}\to V_{gx,gy},g\in G, X$. There is a map from $K$-functor of $Vect[X\times X]^G$ to $\C[X\times X]^G V_{x,y}\to \dim V_{x,y}$. $Vect[X\times X]^G$ has two tensor products: 
\[(V\otimes W)_{x,y}=V_{x,y}\otimes W_{x,y}\quad (V\boxtimes W)_{x,y}=\bigoplus_z V_{x,z}\otimes W_{z,y}.\] Involutions act by $\Pi(V)_{x,y}=\overline{V}_{x,y}, \mathrm{M}(V)_{x,y}=\overline{V}_{y,x}$. If $H,G$ is a Gelfand pair it is possible to choose a tensor subcategory $F[X\times X]^G$ in $Vect[X\times X]^G$ in such a way that $\dim: K(F[X\times X]^G)\otimes \C \to \C[X\times X]^G $ is an isomorphism. Is it true that $Vect[X\times X]^G$, or, perhaps, better $F[X\times X]^G$, is a symmetric monoidal category? What is the formula for commutativity morphisms? Is it possible to categorify duality? Naively, categorification should take a form of a tensor functor
$D:F[X\times X]^G\to F[\check{X}\times \check{X}]^{\check{G}}$,
which interchanges tensor structures and involutions $\Pi$ and $\mathrm{M}$. On the level of $K$-functors it should define an isomorphism $\C[X\times X]^G \to \C[\check{X}\times \check{X}]^{\check{G}}$ that interchanges the algebra structures. The problem has no solution the way it is formulated. The simplest counterexample is $G=\Z_5,H=\{1\}$. We know that vector spaces give us categorification of integers. In order to accommodate entries of matrix $C$ in the categorical language we have to categorify the ring of integers of the cyclotomic field and use it to properly modify $F[X\times X]^G$. Still, even in the case that $C$ has integral coefficients (e.g. in case of $S_{n-1},S_n$) categorification is unknown.
Such a categorification, if it exist, would resemble a construction from \cite{MirkoviandVilonen}. Note that formulas (\ref{E:psidef1}) bases $\{\gX_i\}$ and $\{\Phi_j\}$ have analogues in the Geometric Langlands theory. Element $\gX_i$ is similar to the image in the $K$-theory of a skyscraper sheaf with support on the strata of the affine Grassmannian. $\Phi_j$ is analogous to the image of a perverse sheaf. The main distinction from the Geometric Langlands theory is that our matrix $C$ is never upper-triangular.

\subsection{On the sufficient conditions for existence of the dual pair} The author is not aware of any such a condition except a condition to be  a finite abelian group. Experimental data (about hundred examples) doesn't contradict the following conjecture: If the matrix $C$ corresponding to Gelfand pair $H,G$ satisfying (\ref{E:integrality2}) and (\ref{E:integrality4}) then there is a dual pair $\check{H},\check{G}$. 

\subsection{On the String Topology}
Not too many examples of topological Gelfand pairs discussed in the introduction are known. The most basic is $\{pt\}\subset S^{2n+1}$. $\mathbb{H}_{*}(N,M), \mathbb{H}^{*}(N,M)$ in this case are polynomial algebras in one variable of degree $2n$. Simply-connected compact Lie groups give another class of examples. Again, $N=\{1\}$ is one-point set. A product of such manifolds, we denote it by $P$, can be interpreted as a topological analogue of a finite abelian group (it is also a self-dual object). Let us consider a fiber bundle $M$ with a base $N$ and a fiber $P$. Let us also assume that the bundle has a section $\sigma:N\to M$. This structure is analogous to a semidirect product $A\rtimes H$ from Section \ref{S:semidirect}. We conjecture that $(N,M)$ is a topological Gelfand pair. Under what conditions it is self-dual?

The author is planning to address these questions in the following publications.


\begin{thebibliography}{10}

\bibitem{BreuerLux}
T.~Breuer and K.~Lux.
\newblock The multiplicity-free permutation characters of the sporadic simple
  groups and their automorphism groups.
\newblock {\em Communications in Algebra}, 24(7):2293--2316, 1996.

\bibitem{RonaldBrown}
R.~Brown.
\newblock {\em Topology and Groupoids}.
\newblock BookSurge Publishing, 2006.

\bibitem{ChasSullivan}
M.~Chas and D.~P. Sullivan.
\newblock String topology.
\newblock arXiv.org/abs/math/9911159.

\bibitem{Conway}
J.H. Conway.
\newblock Three lectures on exceptional groups.
\newblock In M.B. Powell and G.~Higman, editors, {\em Finite simple groups},
  pages 215--247. Academic Press, 1971.

\bibitem{CurtisReiner}
C.~W. Curtis and I.~Reiner.
\newblock {\em Representation Theory Of Finite Groups And Associative
  Algebras}.
\newblock Interscience Publishers, 1962.

\bibitem{Egge}
E.~S. Egge.
\newblock A generalization of the \uppercase{T}erwilliger \uppercase{A}lgebra.
\newblock {\em Journal of Algebra}, 233(1):213--252, 1 November 2000.

\bibitem{EtingofGelakiNikshychOstrik}
P.~Etingof, S.~Gelaki, D.~Nikshych, and V.~Ostrik.
\newblock {\em Tensor Categories}, volume 205 of {\em Mathematical Surveys and
  Monographs}.
\newblock AMS, 2015.

\bibitem{FrenkelLanglands}
E.~Frenkel.
\newblock Lectures on the \uppercase{L}anglands \uppercase{P}rogram and
  \uppercase{C}onformal \uppercase{F}ield \uppercase{T}heory.
\newblock In P.~Cartier, P.~Moussa, B.~Julia, and P.~Vanhove, editors, {\em
  Frontiers in Number Theory, Physics, and Geometry II}, pages 387--533.
  Springer, Berlin, Heidelberg, 2007.

\bibitem{HingstonandOancea}
N.~Hingston and A.~Oancea.
\newblock The space of paths in complex projective space with real boundary
  conditions.
\newblock arXiv:1311.7292 [math.GT].

\bibitem{Kirillov}
A.~A. Kirillov.
\newblock {\em Representation Theory and Noncommutative Harmonic Analysis I},
  volume~22 of {\em Encyclopaedia of Mathematical Sciences}.
\newblock Springer-Verlag, Berlin Heidelberg GmbH, 1994.

\bibitem{LupercioUribeaXicotencatl}
E.~Lupercio, B.~Uribe, and M.~A Xicotencatl.
\newblock Orbifold string topology.
\newblock {\em Geometry and Topology}, 12:2203--2247, 2008.

\bibitem{MirkoviandVilonen}
I.~Mirkovi{\'c} and K.~Vilonen.
\newblock Geometric \uppercase{L}anglands duality and representations of
  algebraic groups over commutative rings.
\newblock {\em Annals Of Mathematics}, 166(1):95--143, 2007.

\bibitem{Neukirch}
J.~Neukirch.
\newblock {\em Algebraic Number Theory}.
\newblock Springer, 1999.

\bibitem{Pontryagin}
L.D. Pontryagin.
\newblock {\em Topological groups}.
\newblock Princeton Univ. Press, 1958.

\bibitem{Sullivan2}
D.~Sullivan.
\newblock Open and closed string field theory interpreted in classical
  algebraic topology.
\newblock In {\em Topology, geometry and quantum field theory}, volume 308 of
  {\em London Math. Soc. Lecture Note Ser.}, pages 344--357. Cambridge Univ.
  Press, Cambridge, 2004.

\bibitem{VinbergCom}
E.B. Vinberg.
\newblock Commutative homogeneous spaces and coisotropic actions.
\newblock {\em UMN}, 56(1):3--62, 2001.

\bibitem{MWildon}
M.~Wildon.
\newblock Multiplicity-free representations of symmetric groups.
\newblock {\em Journal of Pure and Applied Algebra}, 213:1464--1477, 2009.

\end{thebibliography}

\end{document}